\documentclass[onefignum,onetabnum]{siamart251216}
\usepackage{amsfonts}
\usepackage{bm}
\usepackage{algorithm}
\usepackage{algpseudocode}
\usepackage{placeins}
\usepackage{url}
\usepackage{listings}
\usepackage{xspace}
\usepackage{subcaption}

\newcommand{\dd}{\,\mathrm{d}}
\newcommand{\ii}{\mathrm{i}}
\newcommand{\dt}{\Delta t}

\newcommand{\dV}{\Delta V}

\newcommand{\kappaC}{\kappa}
\newcommand{\xx}{\bm{x}}
\newcommand{\vv}{\bm{v}}
\newcommand{\VV}{\bm{V}}
\newcommand{\pp}{\bm{p}}
\newcommand{\PP}{\bm{P}}
\newcommand{\AV}{\bm{A}}
\newcommand{\uu}{\bm{u}}
\newcommand{\UU}{\bm{U}}
\newcommand{\JJ}{\bm{J}}
\newcommand{\EE}{\bm{E}}
\newcommand{\BB}{\bm{B}}

\newcommand{\ttau}{\bm{\tau}}
\newcommand{\grad}{\nabla}
\newcommand{\curl}{\nabla\times}
\newcommand{\divg}{\nabla\cdot}
\newcommand{\gradh}{\nabla_h}
\newcommand{\curlh}{\nabla_h\times}
\newcommand{\divgh}{\nabla_h\cdot}

\newcommand{\half}{n+\frac12}
\newcommand{\inner}[2]{\left\langle #1,#2\right\rangle_h}

\newcommand{\DA}{\mathsf{D}}
\newcommand{\KA}{\mathsf{K}}
\newcommand{\Sbar}{\overline{S}}
\newcommand{\vbar}{\overline{\vv}}
\newcommand{\Ubar}{\overline{\UU}}
\newcommand{\Ebar}{\overline{\EE}}

\newcommand{\gphibar}{\overline{\grad\phi}}
\newcommand{\Ah}{\AV_h}

\newcommand{\Wh}{W}

\newcommand{\Kkin}{\mathcal K}

\newcommand{\Etot}{\mathcal{E}}

\newcommand{\bt}{\bm t}
\newcommand{\bs}{\bm s}
\newcommand{\bR}{\bm R}

\newcommand{\maybegraphics}[3][]{%
  \IfFileExists{#2}{\includegraphics[#1]{#2}}{%
    \fbox{\parbox[c][#3][c]{0.92\linewidth}{\centering\small Placeholder for \texttt{\detokenize{#2}}.\\[0.5em]
    Copy the diagnostic figure generated by the C++/Python/Jupyter workflow into the \texttt{figures/} directory.}}}}
\lstdefinestyle{deck}{basicstyle=\ttfamily\scriptsize,breaklines=true,columns=fullflexible,frame=single}

\ifdefined\newsiamremark
  \newsiamremark{remark}{Remark}
\fi

\title{An Energy-Conserving Unstaggered Electromagnetic-Potential Particle-in-Cell Method, Part II: Fully Relativistic Formulation with an Adapted Higuera--Cary Pusher}

\author{Andrew J. Christlieb\thanks{Department of Computational Mathematics, Science and Engineering, Michigan State University, East Lansing, MI.}\and Luis Chac\'{o}n\thanks{Los Alamos National Laboratory, Los Alamos, NM}\and Sining Gong\thanks{Corresponding author. Department of Computational Mathematics, Science and Engineering, Michigan State University, East Lansing, MI}(\email{gongsini@msu.edu}).}

\headers{Energy-Conserving Relativistic Particle Methods}{A. J. Christlieb, L. Chac\'{o}n, and S. Gong}

\usepackage{xcolor}
\author{Andrew J. Christlieb\thanks{Department of Computational
Mathematics, Science and Engineering, Michigan State University, East Lansing,
MI 48824, USA.\funding{This author's work was supported by
AFOSR grant FA9550-24-1-0254, DOE grant DE-SC0023164, DOE/NNSA grant
DE-NA0004265, and ONR grant N00014-24-1-2242.}}
\and Luis Chac\'{o}n\thanks{Los Alamos National Laboratory (retired),
Los Alamos, NM 87545, USA.\funding{This author's contribution
was funded by the Office of Advanced Scientific Computing Research of the
U.S. Department of Energy and performed at Los Alamos National Laboratory
under contract 89233218CNA000001.}}
\and Sining Gong\thanks{Corresponding author. Department of
Computational Mathematics, Science and Engineering, Michigan State University,
East Lansing, MI 48824, USA (\email{gongsini@msu.edu}).}}

\begin{document}
\maketitle

\begin{abstract}
We develop a fully relativistic extension of the energy-conserving, unstaggered electromagnetic-potential particle-in-cell method introduced in our earlier nonrelativistic work.  As in the nonrelativistic generalized-momentum formulation, the fields are advanced by a Crank--Nicolson discretization of the Lorenz-gauge potential equations, the current is deposited from particle orbits, and charge is advanced through the discrete continuity equation.  The new ingredient is a generalized-momentum adaptation of the Higuera--Cary relativistic particle pusher.  The vector-potential force is decomposed into an orbit-discrete-gradient part, which enforces the same finite-difference chain rule used in our earlier nonrelativistic formulation, and a skew part, which generates a Higuera--Cary magnetic rotation in mechanical momentum.  For prescribed fields and frozen orbit data, this generalized-momentum Higuera--Cary map is a unit-Jacobian conjugacy of the mechanical Higuera--Cary map and preserves particle phase volume in $(\xx,\PP)$.  For the fully coupled implicit step, the same orbit data define the current deposit, mesh-to-particle gather, vector-potential chain rule, and relativistic kinetic-energy secant velocity.  The resulting particle work cancels the Crank--Nicolson field-energy balance, yielding exact relativistic total-energy conservation up to nonlinear solver tolerance, orbit-quadrature error, and roundoff. A self-adjoint Fourier projection removes modes on even-grid Nyquist planes while preserving the particle--grid work identity. Numerical tests are organized around cold relativistic two-stream and Weibel/filamentation instabilities, with diagnostics for energy drift, Gauss's law, the Lorenz gauge, particle--grid work, and the orbit chain rule.
\end{abstract}



\section{Introduction}
\label{sec:introduction}

Particle-in-cell (PIC) methods approximate kinetic plasma dynamics by coupling Lagrangian particles to Eulerian electromagnetic fields.  Their long-time reliability depends on more than formal consistency: the current deposition must be compatible with charge conservation, the field update must preserve the electromagnetic involutions, the particle-to-mesh and mesh-to-particle maps must represent the same exchange of work, and the particle pusher should respect the geometric structure of the Vlasov--Maxwell system.  These considerations have shaped electromagnetic PIC from the classical charge-conserving algorithms \cite{BirdsallLangdon1985,HockneyEastwood1988,VillasenorBuneman1992,Esirkepov2001,Verboncoeur2005} through modern implicit and structure-preserving formulations.

Energy conservation has a particularly long history in plasma simulation.  The foundational work of Lewis and Langdon identified discrete particle-field work as the mechanism behind numerical heating and total-energy balance \cite{Lewis1970,Langdon1973}.  Direct-implicit and implicit-moment methods relaxed explicit stability restrictions and motivated fully implicit PIC formulations \cite{Mason1981,CohenLangdonFriedman1982,BrackbillForslund1982,LangdonCohenFriedman1983}.  Later energy-conserving implicit PIC methods sharpened the fully discrete principle: the current appearing in the field equation and the particle work appearing in the kinetic-energy equation must be two representations of the same orbit-averaged exchange \cite{ChenChaconBarnes2011,MarkidisLapenta2011}.  That principle has since been extended to electromagnetic, mapped-mesh, Darwin, boundary-condition, semi-implicit, relativistic, and local-conservation settings \cite{ChaconChen2016,ChaconChen2019,BarnesChacon2021,ChenChaconYin2020,Lapenta2017,ChenToth2019,CamposPintoPages2022,ChaconChen2025}.

A complementary line of work develops geometric and variational PIC methods.  Variational electromagnetic PIC derives semidiscrete or fully discrete equations from an action principle, so that discrete gauge symmetry implies a discrete Gauss law \cite{SquireQinTang2012,EvstatievShadwick2013}.  GEMPIC and compatible finite-element formulations preserve Hamiltonian, Poisson, or Casimir structure and can be combined with energy-conserving time propagation \cite{XiaoQinLiu2015,KrausKormannMorrisonSonnendrucker2017,KormannSonnendrucker2021,CamposPintoKormannSonnendrucker2022}.  Recent explicit and nearly explicit approaches show that exact work balance can also be enforced through constrained splittings or local corrections, including in relativistic regimes \cite{Gonoskov2024,RicketsonHu2025}.  These developments frame the present work: we seek a relativistic PIC method that preserves the potential formulation, the Lorenz-gauge/Gauss-law structure, exact total energy, and the favorable phase-volume property of a relativistic magnetic rotation.

The method builds on the generalized-momentum potential formulation of Christlieb, Sands, and White \cite{ChristliebSandsWhitePartI,ChristliebSandsWhitePartII,ChristliebSandsWhitePartIII}.  In that formulation particles are advanced with the generalized momentum
\begin{equation*}
  \PP=\pp+q\Ah(\xx,t).
\end{equation*}
%Under this representation, 
The update for $\PP$ is expressed in a form such that the canonical force only contains spatial derivatives of the potentials.  This avoids an explicit temporal approximation of $\partial_t\AV$ in the particle force.  The current is mapped first, and charge is then advanced from a compatible discrete continuity equation.  This current-then-continuity source ordering propagates the Lorenz gauge and Gauss's law in the unstaggered potential formulation.  The work in \cite{ChristliebChaconGong2026EC}, hereafter referred to as \textbf{Part I}, added orbit-averaged scatter/gather maps and an orbit-discrete-gradient of $\Ah$ so that the particle-side vector-potential chain rule is exact along each orbit.  That construction yields exact total-energy conservation for the nonrelativistic fully converged Crank--Nicolson (CN) fixed point method, which serves as a foundation for this work.

The present paper extends the construction in \cite{ChristliebChaconGong2026EC} to fully relativistic particles.  The main new step is to replace the additive magnetic part of the generalized-momentum (GM) push by a Higuera--Cary (HC) rotation.  The HC pusher is a relativistic, second-order, volume-preserving charged-particle integrator that recovers the correct relativistic $\EE\times\BB$ drift in the comparisons of Higuera and Cary \cite{HigueraCary2017}; its robustness is part of the broader volume-preserving viewpoint on Boris-type methods \cite{QinZhangXiao2013,HeZhouSun2015}.  In generalized momentum, the orbit-discrete-gradient object from Part I is kept in the operator-applied-to-field form $\DA_i\AV$.  Its transpose is decomposed as
\begin{subequations}
    \begin{align}
    (\DA_i\AV)^T\vbar_i
  & =(\DA_i\AV)\vbar_i+(\KA_i\AV)\vbar_i, \label{eq:D-split-force} \\
  \KA_i\AV & =(\DA_i\AV)^T-\DA_i\AV,  
  \;\;
  (\KA_i\AV)^T =-(\KA_i\AV). \label{eq:D-split-skew} 
\end{align}\label{eq:D-split}
\end{subequations}
The $(\DA_i\AV)\vbar_i$ part participates in the endpoint transformation and the discrete cancellation of $\AV_t$, while the skew part $(\KA_i\AV)\vbar_i$ is used as the magnetic generator in the adapted HC rotation.  For prescribed fields and frozen orbit data, this gives a unit-Jacobian generalized-momentum map in $(\xx,\PP)$.  For the fully coupled implicit PIC step, the same orbit weights, the same chain-rule matrix, and a relativistic kinetic-energy secant velocity give exact total-energy conservation.

The rest of the paper is organized as follows.  \Cref{sec:potential} states the relativistic potential model and the orbit-averaged particle-mesh discretization inherited from Part I.  \Cref{sec:gmhc-volume} defines the generalized-momentum HC particle map and proves the fixed-field particle-volume result.  \Cref{sec:energy} proves the fully coupled relativistic energy balance.  \Cref{sec:results} presents  the results for the cold relativistic two-stream instability and cold relativistic  Weibel instability tests, together with the diagnostics needed to verify energy conservation, Gauss's law, the Lorenz gauge, particle-grid work, and the orbit chain rule.

\section{Relativistic model and orbit-averaged discretization}
\label{sec:potential}

This section fixes the notation used throughout the paper and separates the parts of the method that are inherited from Part I from the new relativistic particle update introduced in \cref{sec:gmhc-volume}.  The field solve, current-then-continuity source ordering, symmetric orbit-averaged particle-mesh maps, and orbit-discrete-gradient chain rule are inherited from our nonrelativistic companion paper \cite{ChristliebChaconGong2026EC}. The present paper keeps those devices and changes the particle mechanics: the nonrelativistic midpoint generalized-momentum push is replaced by a fully relativistic generalized-momentum Higuera--Cary update.

\subsection{Nondimensional relativistic potential model}
\label{subsec:rel-potential-model}

We use the same nondimensional formulation that was introduced in \cite{ChristliebSandsWhitePartI}.  Let
$
  x=L\tilde x,\,
  t=T\tilde t,
\,
  v=\frac{L}{T}\tilde v,
\,
  \phi=\phi_0\tilde\phi,
\,
  \AV=A_0\tilde\AV,
$
with
$
  \phi_0=\frac{ML^2}{QT^2},
 \,
  A_0=\frac{ML}{QT},
\,
  \kappaC=\frac{cT}{L}.
$
Tildes are dropped below.  In the Lorenz gauge, the nondimensional potential equations are
\begin{subequations}
\label{eq:potential-wave}
\begin{align}
  \frac{1}{\kappaC^2}\partial_{tt}\phi-\Delta\phi &= \sigma_1\rho,\label{eq:phi-wave}\\
  \frac{1}{\kappaC^2}\partial_{tt}\AV-\Delta\AV &= \sigma_2\JJ,\label{eq:A-wave}\\
  \frac{1}{\kappaC^2}\partial_t\phi+\divg\AV &=0.\label{eq:lorenz}
\end{align}
\end{subequations}
Here $\kappaC$ is the nondimensional speed of light, and $\sigma_1$ and $\sigma_2$ are the source coefficients determined by the reference scales.  The Maxwell-compatible scaling used in Part I satisfies
\begin{equation}
  \sigma_1=\kappaC^2\sigma_2.
  \label{eq:Maxwell_scaling}
\end{equation}
The algebraic energy argument below only requires that the coefficients used in the potential update, the midpoint Maxwell form, and the field energy be chosen consistently.
The nondimensional electromagnetic fields are recovered from the potentials by
\begin{equation}
  \EE=-\grad\phi-\partial_t\AV,
  \qquad
  \BB=\curl\AV.
  \label{eq:E-B-cont}
\end{equation}
Thus $\divg\BB=0$ is built into the representation. % In the normalized numerical examples of \cref{sec:results}, we take $c=\kappaC=1$, $\epsilon_0=1$, $\mu_0=1$, and $n_0=1$, so that the total electron plasma frequency is $\omega_p^2=n_0e^2/(m_e\epsilon_0)=1$ for the normalized electron charge-to-mass ratio $q/m=-1$.

\subsection{Relativistic generalized-momentum particles}
\label{subsec:rel-particles}

For particle $i$, let $m_i$ and $q_i$ denote mass and charge.  The mechanical momentum, Lorentz factor, and velocity are
\begin{equation}
  \pp_i=m_i\gamma_i(\vv_i)\vv_i,
  \qquad
  \gamma_i(\vv)=\frac{1}{\sqrt{1-|\vv|^2/\kappaC^2}},
  \qquad
  \vv_i=\VV_i(\pp_i)=\frac{\pp_i}{m_i\gamma_i(\pp_i)},
  \label{eq:rel-particle-def}
\end{equation}
where
  $\gamma_i(\pp)=\sqrt{1+\frac{|\pp|^2}{m_i^2\kappaC^2}}.$
The generalized momentum is
\begin{equation}
  \PP_i=\pp_i+q_i\AV(\xx_i,t).
  \label{eq:rel-P-def}
\end{equation}
In canonical variables the relativistic Hamiltonian is
\begin{equation}
  H_i(\xx_i,\PP_i,t)
  =m_i\kappaC^2\left(\gamma_i\bigl(\PP_i-q_i\AV(\xx_i,t)\bigr)-1\right)
   +q_i\phi(\xx_i,t).
  \label{eq:rel-Hamiltonian}
\end{equation}
Hamilton's equations give
\begin{subequations}
\label{eq:canonical-continuous}
\begin{align}
  \dot{\xx}_i &= \vv_i = \VV_i\bigl(\PP_i-q_i\AV(\xx_i,t)\bigr),\label{eq:xdot-canonical}\\
  \dot{\PP}_i &= -q_i\grad\phi(\xx_i,t)+q_i\bigl(\grad\AV(\xx_i,t)\bigr)^T\vv_i.\label{eq:Pdot-canonical}
\end{align}
\end{subequations}
We note that the time derivative of $\AV$ is not discarded.  It cancels when the mechanical Lorentz equation is written in the generalized momentum variable:
\begin{equation*}
\begin{split}
  \dot{\PP}_i
  &=\dot{\pp}_i+q_i\frac{\dd}{\dd t}\AV(\xx_i(t),t) \\
  &=q_i\bigl(-\grad\phi-\partial_t\AV+\vv_i\times\curl\AV\bigr)
    +q_i\bigl(\partial_t\AV+(\grad\AV)\vv_i\bigr)  \\
  &=-q_i\grad\phi+q_i(\grad\AV)^T\vv_i .
\end{split}
\label{eq:At-cancellation}
\end{equation*}
This cancellation is the generalized-momentum advantage that Part I exploited in the nonrelativistic setting.  In the relativistic method, the same cancellation is retained, but the magnetic part of the particle update is advanced by an adapted Higuera--Cary rotation rather than by an additive midpoint force.  This is the subject of sections 3 and 4.

\subsection{Mesh-based field values and the Crank--Nicolson potential update}
\label{subsec:mesh-fields}

Throughout the paper, a \emph{mesh field} means a scalar or vector array carried on the mesh.  Thus $\phi_g$, $\rho_g$, $\AV_g$, $\UU_g$, $\JJ_g$, $\EE_g$, and $\BB_g$ are all mesh-based values, even though only $\EE$ and $\BB$ are electromagnetic fields in the physical sense.  This convention is useful because the particle-mesh maps below apply to any mesh-based value.  The same notation as in Part I is used for the mesh inner product and cell volume,
\begin{equation}
  \inner{F}{G}=\sum_g F_g\cdot G_g\dV.
\end{equation}
The spatial operators $\gradh$, $\divgh$, $\curlh$, and $\Delta_h$ are assumed to satisfy the usual periodic summation-by-parts identities.  In the implementation described here they are spectral operators on a periodic mesh \cite{GottliebOrszag1977,Boyd2001,Trefethen2000,CanutoHussainiQuarteroniZang2006}.

Introduce the auxiliary variables $ \psi=\partial_t\phi,\, \UU=\partial_t\AV.$
The field step is the same Crank–Nicolson update of the first-order potential system used in Part I:
\begin{subequations}
\label{eq:CN-field}
\begin{align}
  \frac{\phi^{n+1}-\phi^n}{\dt} &= \psi^{\half},\label{eq:cn-phi}\\
  \frac{\psi^{n+1}-\psi^n}{\dt} &= \kappaC^2\Delta_h\phi^{\half}+\kappaC^2\sigma_1\rho^{\half},\label{eq:cn-psi}\\
  \frac{\AV^{n+1}-\AV^n}{\dt} &= \UU^{\half},\label{eq:cn-A}\\
  \frac{\UU^{n+1}-\UU^n}{\dt} &= \kappaC^2\Delta_h\AV^{\half}+\kappaC^2\sigma_2\JJ^{\half}.\label{eq:cn-U}
\end{align}
\end{subequations}
Here $f^{\half}=(f^{n+1}+f^n)/2$.  The current is deposited from the accepted particle orbits first, and charge is then advanced by the same discrete divergence used by the potential equations:
\begin{equation}
  \frac{\rho^{n+1}-\rho^n}{\dt}+\divgh\JJ^{\half}=0,
  \qquad
  \rho^{\half}=\frac{1}{2}(\rho^{n+1}+\rho^n).
  \label{eq:cn-continuity}
\end{equation}
With the midpoint Lorenz gauge, 
\begin{equation}\label{eqn:gauge conditions}
    \psi^{n+1/2}/\kappa^2+\nabla_h\cdot \AV^{\half}=0,
\end{equation}
this current-then-continuity ordering is the discrete mechanism that propagates the Lorenz gauge and Gauss's law provided the initial discrete Lorenz and Gauss constraints are satisfied.  This source ordering was developed in the generalized-momentum potential PIC papers and is retained unchanged here.

The discrete electromagnetic fields used in the energy balance are
\begin{equation}
  \EE_g^n=-\gradh\phi_g^n-\UU_g^n,
  \qquad
  \BB_g^n=\curlh\AV_g^n.
  \label{eq:fields-from-potentials-discrete}
\end{equation}
The compatible field energy is
\begin{equation}
  W^n=\sum_g\left(\frac{1}{2\sigma_1}|\EE_g^n|^2+
  \frac{1}{2\sigma_2}|\BB_g^n|^2\right)\dV.
  \label{eq:field-energy-section2}
\end{equation}

\subsection{Symmetric orbit-averaged maps between particles and mesh}
\label{subsec:orbit-averaged-maps}

Let $S_g(\xx)$ be the periodic particle shape associated with mesh point $g$.  A mesh vector potential is gathered to a particle through the interpolant
\begin{equation}
  \Ah^n(\xx)=\sum_g \AV_g^n S_g(\xx).
  \label{eq:mesh-interpolant}
\end{equation}
The same shape family is used in both directions.  The difference from a standard midpoint PIC map is that the weights are averaged over the particle orbit rather than evaluated at a single point.

For one step, define the unwrapped straight orbit
\begin{equation}
  \xx_i(s)=\xx_i^n+s\Delta\xx_i,
  \qquad
  0\le s\le 1,
  \qquad
  \Delta\xx_i=\xx_i^{n+1}-\xx_i^n=\dt\,\vbar_i.
  \label{eq:rel-orbit}
\end{equation}
The orbit velocity $\vbar_i$ is determined by the relativistic particle update in \cref{sec:gmhc-volume,sec:energy}; for the present section it is enough that the same $\vbar_i$ defines the orbit, the current, and the work identity.  The orbit-averaged shape is
\begin{equation}
  \Sbar_{ig}=\int_0^1 S_g(\xx_i(s))\dd s.
  \label{eq:orbit-average-shape}
\end{equation}
The particle-to-mesh current map is
\begin{equation}
  \JJ_g^{\half}=\frac{1}{\dV}\sum_i q_i\vbar_i\Sbar_{ig}.
  \label{eq:orbit-current}
\end{equation}
The symmetric mesh-to-particle gather for any mesh-based vector value $\bm{G}_g$ is
\begin{equation}
  \overline{\bm{G}}_i=\sum_g \bm{G}_g\Sbar_{ig}.
  \label{eq:orbit-gather-general}
\end{equation}
Using the same orbit weights in \cref{eq:orbit-current,eq:orbit-gather-general} gives the exact scatter/gather adjointness identity
\begin{equation}
  \sum_i q_i\vbar_i\cdot\overline{\bm{G}}_i
  =\sum_g \JJ_g^{\half}\cdot \bm{G}_g\dV.
  \label{eq:scatter-gather-work}
\end{equation}
This identity is the particle-mesh work equivalence used in the total-energy proof.  It is independent of the relativistic form of the particle pusher; all that matters is that current deposition and field gathering use the same orbit-averaged map.

\subsection{\texorpdfstring{Orbit-averaged discrete chain rule and discrete gradient}{Orbit-averaged discrete chain rule and discrete gradient}}
\label{subsec:orbit-discrete-gradient}

The second Part I device retained here is the exact finite-difference chain rule for the mesh-interpolated vector potential along the same orbit.  Linearly interpolate the mesh vector potential during the step,
\begin{equation}
  \AV_g(s)=(1-s)\AV_g^n+s\AV_g^{n+1},
  \label{eq:A-g-s}
\end{equation}
so that the particle-sampled vector potential along the orbit is
\begin{equation}
  \mathcal A_i(s)=\sum_g \AV_g(s)S_g(\xx_i(s)).
  \label{eq:Acal-s}
\end{equation}
Using \cref{eq:cn-A}, define the orbit-averaged gather of $\UU^{\half}$ by
\begin{equation}
  \Ubar_i=\sum_g \UU_g^{\half}\Sbar_{ig}.
  \label{eq:Ubar-def}
\end{equation}
The orbit-discrete-gradient matrix is
\begin{equation}
  [\DA_i\AV]_{\ell j}
  =\sum_g\int_0^1 A_{\ell,g}(s)\,\partial_{x_j}S_g(\xx_i(s))\dd s,
  \qquad
  1\le \ell,j\le d.
  \label{eq:D-def}
\end{equation}
Equivalently, $\DA_i\AV$ is the spatial Jacobian of the time-interpolated particle interpolant, averaged along the accepted orbit.  It is not an independently gathered mesh gradient.

\begin{lemma}[Exact orbit-averaged chain rule]
\label{lem:exact-chain-rule}
The definitions \eqref{eq:orbit-average-shape}, \eqref{eq:A-g-s}, \eqref{eq:Ubar-def}, and \eqref{eq:D-def} imply
\begin{equation}
  \Ah^{n+1}(\xx_i^{n+1})-\Ah^n(\xx_i^n)
  =\dt\,\Ubar_i+\dt\,(\DA_i\AV)\vbar_i.
  \label{eq:A-chain-rule}
\end{equation}
\end{lemma}

\begin{proof}
Differentiate \eqref{eq:Acal-s} with respect to $s$:
\begin{equation}
  \frac{\dd \mathcal A_i}{\dd s}
  =\sum_g(\AV_g^{n+1}-\AV_g^n)S_g(\xx_i(s))
   +\sum_g\AV_g(s)\grad S_g(\xx_i(s))\cdot\Delta\xx_i .
\end{equation}
Using \eqref{eq:cn-A} and $\Delta\xx_i=\dt\vbar_i$, then integrating from $s=0$ to $s=1$, gives \eqref{eq:A-chain-rule}.
\end{proof}

This chain rule is the algebraic bridge between generalized momentum and energy conservation.  In Part I it cancels the vector-potential contribution in the nonrelativistic kinetic-energy balance.  In the present paper it plays the same role in the relativistic balance, while the skew part of the same orbit-gradient object supplies the effective magnetic generator for the Higuera--Cary rotation.  \textcolor{black}{To make the dependence on the mesh vector potential explicit, we retain the notation $\DA_i\AV$ and $\KA_i\AV$ used in Part I; see
\eqref{eq:D-split}. These expressions represent operators applied to
the mesh vector potential and appear in this form in the relativistic
particle update. } Exact conservation requires the orbit integrals in \cref{eq:orbit-average-shape,eq:D-def} to be evaluated as exact integrals of the pulled-back spline functions.  For compactly supported B-splines these functions are piecewise polynomial in $s$; the implementation therefore splits each orbit at crossed spline knots before applying Gaussian quadrature, as in Part I.

\section{Generalized-momentum Higuera--Cary push and fixed-field volume preservation}
\label{sec:gmhc-volume}

This section isolates the new particle map.  The purpose is to import the favorable volume-preservation property of the relativistic Higuera--Cary (HC) magnetic rotation into the generalized-momentum potential formulation without reintroducing an explicit $A_t$ force gather.  The HC pusher is a second-order leapfrog method for relativistic charged particles: electric forces are applied as two half-impulses, while the magnetic part is applied as a centered rotation in mechanical momentum.  Historically, HC can be viewed as a relativistic member of the Boris family of leapfrog rotation methods \cite{Boris1970,HigueraCary2017}; however, the definitions and proofs below are written directly in HC variables.  

\subsection{HC as a leapfrog rotation map}
\label{subsec:hc-leapfrog-rotation}

For one particle, with the fields held fixed during one particle substep, the relativistic Lorentz equations in mechanical momentum are
\begin{equation}
  \dot{\xx}=\vv(\pp),
  \;\;
  \dot{\pp}=q\bigl(\EE+\vv(\pp)\times\BB\bigr),
  \;\;
  \vv(\pp)=\frac{\pp}{m\gamma(\pp)},
  \;\;
  \gamma(\pp)=\sqrt{1+\frac{|\pp|^2}{m^2\kappaC^2}}.
  \label{eq:rel-lorentz-hc}
\end{equation}
The electric term can change the particle energy.  The magnetic term is different: it is perpendicular to the velocity and therefore should rotate the momentum without changing its magnitude when $\EE=0$.  The HC update is designed to respect this distinction at the discrete level.

In a conventional leapfrog notation the particle position is stored at integer times and the mechanical momentum is stored at half times.  The HC step has the schematic form
% \begin{subequations}
% \label{eq:hc-leapfrog-template}
% \begin{align}
%   \pp^- = \pp^{n-1/2}+\frac{q\dt}{2}\EE^n, & \qquad \pp^+ = \bR_{\rm HC}^{p}(\pp^-;\BB^n),
%   \label{eq:hc-template-rot}\\
%   \pp^{n+1/2} = \pp^++\frac{q\dt}{2}\EE^n, & \qquad \xx^{n+1} = \xx^n+\dt\,\vv(\pp^{n+1/2}).
%   \label{eq:hc-template-drift}
% \end{align}
% \end{subequations}
\begin{subequations}
\label{eq:hc-leapfrog-template}
\begin{align}
  \pp^- & = \pp^{n-1/2}+\frac{q\dt}{2}\EE^n, \label{eq:hc-template-plus} \\ 
  \pp^+ & = \bR_{\rm HC}^{p}(\pp^-;\BB^n),
  \label{eq:hc-template-rot}\\
  \pp^{n+1/2} & = \pp^++\frac{q\dt}{2}\EE^n, \label{eq:hc-template-minus}\\
  \xx^{n+1} & = \xx^n+\dt\,\vv(\pp^{n+1/2}).
  \label{eq:hc-template-drift}
\end{align}
\end{subequations}
The intermediate momenta $\pp^-$ and $\pp^+$ are not additional time levels.  They are the momentum after the first electric half-impulse and after the magnetic rotation, respectively.  In the orbit-averaged method developed below, the last line is replaced by the accepted particle orbit and its averaged velocity $\vbar_i$, but the same kick--rotation--kick momentum geometry is retained.

The magnetic rotation in \cref{eq:hc-template-rot} is the key point.  HC writes the magnetic substep in the centered form
\begin{equation}
  \pp^+-\pp^-
  =q\dt\left(\frac{\pp^++\pp^-}{2m\Gamma}\right)\times\BB,
  \label{eq:hc-centered-rotation-overview}
\end{equation}
where $\Gamma$ is the HC-centered Lorentz factor for this magnetic
rotation.  In implicit centered form it is
\begin{equation}
  \Gamma
  =
  \left(
  1+
  \left|
  \frac{\pp^++\pp^-}{2m\kappaC}
  \right|^2
  \right)^{1/2}.
  \label{eq:Gamma-centered-overview}
\end{equation}
The explicit HC algorithm described in the next subsection computes
this same scalar from $\pp^-$ and $\BB$ before applying the rotation.
For a fixed $\Gamma$, this equation is an exact Cayley rotation in momentum space.  The zero-work property is immediate: taking the dot product of \eqref{eq:hc-centered-rotation-overview} with $\pp^++\pp^-$ gives
\begin{equation}
  |\pp^+|^2-|\pp^-|^2=0,
\end{equation}
because the right-hand side of \cref{eq:hc-centered-rotation-overview} is perpendicular to $\pp^++\pp^-$.  Thus the magnetic substep preserves $|\pp|$, and therefore preserves the relativistic kinetic energy, in a purely magnetic field.  The practical HC formula evaluates this rotation explicitly using two cross products after computing the HC-centered Lorentz factor $\Gamma=\gamma_{\rm HC}$; the complete algebraic definition is given in \cref{def:hc-normalized,def:hc-dimensional}.

The choice of $\Gamma$ is what distinguishes HC from other relativistic leapfrog pushers.  Higuera and Cary choose $\gamma_{\rm HC}$ from the incoming momentum and magnetic field so that the complete nonlinear magnetic map is volume preserving and gives the correct relativistic $\EE\times\BB$ drift behavior \cite{HigueraCary2017,Vay2008}.  

The design principle for integrators of this type is this: electric half-impulses and position drifts are unit-determinant shears, while the HC magnetic update is a unit-Jacobian rotation map in momentum.  Their composition therefore preserves fixed-field particle phase volume.  This property is weaker than symplecticity, but it prevents artificial compression or expansion of small phase-space volumes and is one reason these rotation-based charged-particle methods are robust in long integrations \cite{QinZhangXiao2013,HeZhouSun2015}.

Our generalized-momentum construction keeps this HC leapfrog-rotation structure but generates the rotation from potentials rather than from a directly gathered magnetic field.  The canonical momentum is
\begin{equation}
  \PP=\pp+q\AV.
\end{equation}
Applying the mechanical HC update directly requires the electric field $\EE=-\grad\phi-\AV_t$.  A direct force gather of $\AV_t$ would obscure the energy cancellation developed in Part I.  \textcolor{black}{Instead, we use the orbit-averaged vector-potential chain rule together with the decomposition in \eqref{eq:D-split-force}. This separates the skew magnetic contribution, generated by $\KA_i\AV$ as defined in
\eqref{eq:D-split-skew}, from the discrete-gradient contribution
involving $\DA_i\AV$, which participates in the cancellation of
the explicit $\AV_t$ term.} 
The role of this section is therefore to replace the magnetic rotation generated by $\BB$ in the HC map by the orbit-generated skew matrix $\KA_i\AV$, while keeping the notation explicit as $\DA_i\AV$ and $\KA_i\AV$.

\subsection{Mechanical Higuera--Cary rotation}
Before writing this in generalized momentum, we now review the magnetic-rotation map $\bR_{\rm HC}^p$ appearing in \cref{eq:hc-leapfrog-template} in detail. We then transfer this map to generalized momentum.  All variables are the nondimensional variables fixed in \cref{sec:potential}; in particular the speed of light is $\kappaC$.  The auxiliary variable $u=p/(m\kappaC)$ is used only to write the HC algebra compactly.  The definition is included explicitly so that the later generalized-momentum map contains no hidden particle-pusher conventions.

For a vector $\BB=(B_1,B_2,B_3)^T$, define the skew matrix $C(\BB)$ by
\begin{equation}
  C(\BB)=
  \begin{pmatrix}
  0&B_3&-B_2\\
  -B_3&0&B_1\\
  B_2&-B_1&0
  \end{pmatrix},
  \qquad C(\BB)\vv=\vv\times \BB .
  \label{eq:cross-matrix}
\end{equation}
This sign convention is used throughout the paper.  It differs by a sign from the common convention in which the matrix represents $\BB\times \vv$.

\begin{definition}[Normalized HC magnetic map]
\label{def:hc-normalized}
Fix one species and one time step, so that $q$, $m$, $\kappaC$, and $\dt$ are known.  For a given incoming normalized mechanical momentum
\begin{equation}
  \uu^- = \frac{\pp^-}{m\kappaC},
  \qquad \gamma^- = \gamma(\uu^-)=\sqrt{1+|\uu^-|^2},
\end{equation}
and a fixed magnetic field $\BB$, set
\begin{equation}
  \ttau=\frac{q\dt}{2m}\BB,
  \qquad
  \sigma=(\gamma^-)^2-|\ttau|^2.
  \label{eq:hc-tau-sigma}
\end{equation}
The HC Lorentz factor for the magnetic substep is the positive root
\begin{equation}
  \gamma_{\rm HC}(\uu^-;\BB)
  =\left\{
  \frac{1}{2}
  \left[
  \sigma+\bigl(\sigma^2+4(|\ttau|^2+(\uu^-\cdot\ttau)^2)\bigr)^{1/2}
  \right]
  \right\}^{1/2}.
  \label{eq:gammaHC}
\end{equation}
Define
\begin{equation}
  \bt_{\rm HC}=\frac{\ttau}{\gamma_{\rm HC}},
  \qquad
  \bs_{\rm HC}=\frac{2 \, \bt_{\rm HC}}{1+|\bt_{\rm HC}|^2},
  \label{eq:hc-ts}
\end{equation}
and then perform the two cross-product operations
\begin{subequations}
\label{eq:hc-explicit}
\begin{align}
  \uu^\star &= \uu^-+\uu^-\times \bt_{\rm HC},\label{eq:hc-ustar}\\
  \uu^+ &= \uu^-+\uu^\star\times \bs_{\rm HC}.\label{eq:hc-uplus}
\end{align}
\end{subequations}
The normalized HC magnetic map is the explicitly defined function
\begin{equation}
  \bR_{\rm HC}^{u}(\uu^-;\BB)=\uu^+,
  \label{eq:RHC-u-def}
\end{equation}
where $\uu^+$ is the value produced by \cref{eq:hc-tau-sigma,eq:gammaHC,eq:hc-ts,eq:hc-explicit}.
\end{definition}

\begin{definition}[Nondimensional mechanical-momentum HC magnetic map]
\label{def:hc-dimensional}
For an incoming mechanical momentum $\pp^-$ and fixed magnetic field $\BB$, both in the nondimensional variables of \cref{sec:potential}, the mechanical-momentum HC magnetic map is
\begin{equation}
  \bR_{\rm HC}^p(\pp^-;\BB)
  :=m\kappaC \, \bR_{\rm HC}^u\!\bigl(\pp^-/(m\kappaC);\BB\bigr).
  \label{eq:RHC-p-def}
\end{equation}
% Equivalently, $\bR_{\rm HC}^p(\pp^-;\BB)$ is obtained by applying
% Definition~\ref{def:hc-normalized} to $\uu^-=\pp^-/(m\kappaC)$
% and multiplying the resulting $\uu^+$ by $m\kappaC$.
% The species parameters $q$, $m$, $\kappaC$, and the time step $\dt$ are suppressed in the notation. Using Definition~\ref{def:hc-normalized}, the map can be written explicitly as
% The species parameters and time step are suppressed in the notation;
% the map uses the nondimensional $q$, $m$, $\kappaC$, and $\dt$
% of the particle being advanced.  Thus $\bR_{\rm HC}^p$ is explicitly specified by
% Definition~\ref{def:hc-normalized} together with the scaling
% in \eqref{eq:RHC-p-def}.  The same map in compact form is given by
% \begin{equation}
%   \uu^- = \frac{\pp^-}{m\kappaC},\qquad
%   \bt=\frac{\ttau}{\gamma_{\rm HC}(\uu^-;\BB)},\qquad
%   \bs=\frac{2\,\bt}{1+|\bt|^2},
%   \label{eq:RHC-ts-short}
% \end{equation}
% where $\ttau$ and $\gamma_{\rm HC}$ are given by \cref{eq:hc-tau-sigma,eq:gammaHC}.  Then
The map uses the nondimensional $q$, $m$, $\kappaC$, and $\dt$ of the particle being advanced; these parameters are suppressed in the notation. Definition~\ref{def:hc-normalized} and the scaling in \eqref{eq:RHC-p-def} fully specify the map. For its compact form, set
\begin{equation}
  \uu^- = \frac{\pp^-}{m\kappaC},\qquad
  \bt=\frac{\ttau}{\gamma_{\rm HC}(\uu^-;\BB)},\qquad
  \bs=\frac{2\,\bt}{1+|\bt|^2},
  \label{eq:RHC-ts-short}
\end{equation}
where $\ttau$ and $\gamma_{\rm HC}$ are given by \cref{eq:hc-tau-sigma,eq:gammaHC}. Then
\begin{equation}
\label{eq:RHC-one-line}
  \boxed{
  \begin{aligned}
  \bR_{\rm HC}^p(\pp^-;\BB) &= m\kappaC\,\uu^+, \, 
  \uu^+ = \uu^-+\left(\uu^-+\uu^-\times \bt\right)\times \bs .
  \end{aligned}}
\end{equation}
This boxed expression is the function $\bR_{\rm HC}^p$ used below.
\end{definition}

After the scalar $\gamma_{\rm HC}$ has been computed from \cref{eq:gammaHC}, the same update can be written as a Cayley transform.  With the cross-product convention \cref{eq:cross-matrix},
\begin{equation}
  \bR_{\rm HC}^p(\pp^-;\BB)
  =Q(\BB,\gamma_{\rm HC})\pp^-,
  \;\;
  Q(\BB,\Gamma)=
  \left(I-\frac{q\dt}{2m\Gamma}C(\BB)\right)^{-1}
  \left(I+\frac{q\dt}{2m\Gamma}C(\BB)\right).
  \label{eq:hc-cayley-B}
\end{equation}
Since $C(\BB)^T=-C(\BB)$, $Q(\BB,\Gamma)$ is an orthogonal determinant-one rotation for each fixed positive scalar $\Gamma$.  The determinant-one statement for the full HC magnetic substep should not be inferred only from this fixed-$\Gamma$ Cayley matrix, because $\gamma_{\rm HC}$ depends on $\uu^-$.  Higuera and Cary's result gives the volume-preservation of the complete nonlinear map \cref{eq:RHC-one-line}, including the dependence of $\gamma_{\rm HC}$ on $\uu^-$.  In the energy proof below, only skewness and magnetic zero work are needed; in the volume statement, we invoke the full HC volume-preserving map.

\subsection{Skew-matrix HC map used by the generalized-momentum method}

% The transpose of the orbit-discrete-gradient matrix $\DA_i\AV$ enters the canonical force. We write it as the original matrix plus the skew matrix $\KA_i\AV$, as defined in \cref{eq:D-split}. The first contribution participates in the endpoint chain-rule cancellation, while the skew contribution generates the HC magnetic rotation. This rotation does no work; total-energy conservation additionally requires the electric-work identity and compatible particle--mesh coupling.
% The $(\DA_i\AV)\vbar_i$ term participates in the endpoint transformation and in the cancellation of the explicit $\UU=\partial_t\AV$ contribution through \cref{eq:A-chain-rule}.   The skew term $(\KA_i\AV)\vbar_i$ is the magnetic part and is placed inside an HC rotation.
The transpose of the orbit-discrete-gradient matrix $\DA_i\AV$ enters the canonical force. We write it as the original matrix plus the skew matrix $\KA_i\AV$, as defined in \cref{eq:D-split}. The $(\DA_i\AV)\vbar_i$ term participates in the endpoint transformation and in the cancellation of the explicit $\UU=\partial_t\AV$ contribution through \cref{eq:A-chain-rule}. The skew term $(\KA_i\AV)\vbar_i$ is the magnetic part and generates the HC magnetic rotation. This rotation does no work; total-energy conservation additionally requires the electric-work identity and compatible particle--mesh coupling.

In three dimensions a skew matrix has an axial vector.  We use a generic skew matrix $M$ only for this algebraic axial-vector helper; the orbit object used by the PIC method is obtained by substituting $M=\KA_i\AV$.  The convention is $  M\vv = \vv\times \BB_{\rm eff}(M)$.
Thus, if
\begin{equation}
  M=\begin{pmatrix}
  0&M_{12}&M_{13}\\
  -M_{12}&0&M_{23}\\
  -M_{13}&-M_{23}&0
  \end{pmatrix},
\end{equation}
then $ \BB_{\rm eff}(M)=(M_{23},-M_{13},M_{12})^T, \, C(\BB_{\rm eff}(M))=M.$
For the orbit-generated magnetic matrix, this means
\begin{equation}
  \BB_{{\rm eff},i}(\AV):=\BB_{\rm eff}(\KA_i\AV),
  \qquad
  (\KA_i\AV)\vv=\vv\times \BB_{{\rm eff},i}(\AV).
  \label{eq:beff-KA}
\end{equation}
The sign convention is checked by the gauge $\AV=(0,B x,0)$, for which the skew orbit-gradient part reduces to $\KA_i\AV=(\grad\AV)^T-\grad\AV$ and gives $\BB_{{\rm eff},i}(\AV)=(0,0,B)^T$ and $(\KA_i\AV)\vv=\vv\times B \hat{\bm z}$.

\begin{definition}[HC map generated by the orbit skew matrix $\KA_i\AV$]
\label{def:hc-skew}
For the orbit-generated skew matrix $\KA_i\AV$, the HC mechanical-momentum map is
\begin{equation}
  \bR_{\rm HC}^p(\pp^-;\KA_i\AV)
  :=\bR_{\rm HC}^p\!\bigl(\pp^-;\BB_{{\rm eff},i}(\AV)\bigr).
  \label{eq:RHC-K-def}
\end{equation}
Equivalently, substitute $\BB=\BB_{{\rm eff},i}(\AV)$ into the explicit formula \cref{eq:RHC-one-line}.  In expanded form,
\begin{equation}
  \bR_{\rm HC}^p(\pp^-;\KA_i\AV)
  =m\kappaC\left[\uu^-+\left(\uu^-+\uu^-\times \bt_i^{\AV}\right)\times \bs_i^{\AV}\right],
  \label{eq:RHC-K-one-line}
\end{equation}
where we use Definition~\ref{def:hc-normalized} with
$\BB=\BB_{{\rm eff},i}(\AV)$. 
% \begin{equation}
% \begin{aligned}
%   \uu^- = \frac{\pp^-}{m\kappaC}, \,
%   \ttau_i^{\AV} =\frac{q\dt}{2m}\BB_{{\rm eff},i}(\AV), \, 
%   \bt_i^{\AV} =\frac{\ttau_i^{\AV}}{\gamma_{\rm HC}(\uu^-;\BB_{{\rm eff},i}(\AV))}, \, 
%   \bs_i^{\AV} =\frac{2\bt_i^{\AV}}{1+|\bt_i^{\AV}|^2}.
% \end{aligned}
%   \label{eq:RHC-K-data}
% \end{equation}
Thus $\bR_{\rm HC}^p(\pp^-;\KA_i\AV)$ is a completely specified algebraic map: build $\BB_{{\rm eff},i}(\AV)$ from $\KA_i\AV$, compute the HC scalar with that effective magnetic field, and evaluate \cref{eq:RHC-K-one-line}.
Equivalently, after $\gamma_{\rm HC}=\gamma_{\rm HC}(\pp^-/(m\kappaC);\BB_{{\rm eff},i}(\AV))$ has been computed from \cref{eq:gammaHC},
\begin{equation}
\begin{aligned}
  \bR_{\rm HC}^p(\pp^-;\KA_i\AV)&=\bR_i^{\AV}(\gamma_{\rm HC})\pp^-,\\
  \bR_i^{\AV}(\Gamma)&=
  \left(I-\frac{q\dt}{2m\Gamma}\KA_i\AV\right)^{-1}
  \left(I+\frac{q\dt}{2m\Gamma}\KA_i\AV\right).
\end{aligned}
  \label{eq:RK-def}
\end{equation}
The HC-centered velocity associated with this magnetic rotation is
\begin{equation}
  \VV_{\rm HC}(\pp^+,\pp^-;\KA_i\AV)
  =\frac{\pp^+ + \pp^-}{2m\,\gamma_{\rm HC}(\pp^-/(m\kappaC);\BB_{{\rm eff},i}(\AV))},
  \qquad \pp^+=\bR_{\rm HC}^p(\pp^-;\KA_i\AV).
  \label{eq:VHC-def}
\end{equation}
With these definitions the Cayley equation is equivalent to
\begin{equation}
  \pp^+ - \pp^- = q\dt\,(\KA_i\AV)\,\VV_{\rm HC}(\pp^+,\pp^-;\KA_i\AV).
  \label{eq:hc-magnetic-increment}
\end{equation}
\end{definition}

% In the rest of the paper, the shorter notation $R_{\rm HC}(\pp^-;\KA_i\AV)$ means the mechanical-momentum map $\bR_{\rm HC}^p(\pp^-;\KA_i\AV)$ in the nondimensional variables of \cref{sec:potential}.  Likewise $R_{\rm HC}(\uu^-;\BB)$, when used in a discussion of normalized variables, means the normalized map $\bR_{\rm HC}^u$ in \cref{eq:RHC-u-def}.  These are not additional undefined functions; they are abbreviations for the explicit formulas above.

\subsection{No-explicit-\texorpdfstring{$A_t$}{A-t} generalized-momentum HC step}

Let
\begin{equation}
  \AV_i^- = \Ah^n(\xx_i^n),\quad
  \AV_i^+=\Ah^{n+1}(\xx_i^{n+1}),\quad
  \overline{\AV}_{{\rm ep},i}=\frac{1}{2}(\AV_i^-+\AV_i^+),
\end{equation}
and gather the scalar-potential force with the orbit weights,
\begin{equation}
  \gphibar_i=\sum_g (\gradh\phi_g^{\half})\Sbar_{ig}.
\end{equation}
The canonical no-explicit-$\AV_t$ half transforms are
\begin{subequations}
\label{eq:gmhc-step}
\begin{align}
  \pp_i^- & = \PP_i^n-q_i\overline{\AV}_{{\rm ep},i}
          -\frac{q_i\dt}{2}\gphibar_i
          +\frac{q_i\dt}{2}(\DA_i\AV)\vbar_i,\label{eq:gmhc-pre}\\
  \pp_i^+ & = \bR_{\rm HC}^p(\pp_i^-;\KA_i\AV),\label{eq:gmhc-rot}\\
  \PP_i^{n+1} & = \pp_i^+ +q_i\overline{\AV}_{{\rm ep},i}
          -\frac{q_i\dt}{2}\gphibar_i
          +\frac{q_i\dt}{2}(\DA_i\AV)\vbar_i.\label{eq:gmhc-post}
\end{align}
\end{subequations}
Here $\bR_{\rm HC}^p(\pp^-;\KA_i\AV)$ is exactly the skew-matrix HC map in Definition~\ref{def:hc-skew}. 
% build $B_{{\rm eff},i}(\AV)=B_{\rm eff}(\KA_i\AV)$ from \cref{eq:beff-KA}, compute $\gamma_{\rm HC}$ from \cref{eq:gammaHC}, and apply the cross-product update \cref{eq:hc-explicit} or, equivalently, the Cayley matrix \cref{eq:RK-def}.  
The plus sign in front of $(\DA_i\AV)\vbar_i$ in both canonical half transforms is the discrete signature of the $\AV_t$ cancellation.  Indeed, by \cref{eq:A-chain-rule},
\begin{equation}
  -\frac{q_i\dt}{2}\Ubar_i
  = -\frac{q_i}{2}\left(\Ah^{n+1}(\xx_i^{n+1})-\Ah^n(\xx_i^n)\right)
    +\frac{q_i\dt}{2}(\DA_i\AV)\vbar_i.
\end{equation}
Thus \cref{eq:gmhc-step} is algebraically equivalent to mechanical electric half-kicks with $\Ebar_i=-\gphibar_i-\Ubar_i$, but the particle force never uses an explicit finite-difference approximation to $\AV_t$.

\subsection{Particle-volume theorem for prescribed fields}

The volume-preservation theorem concerns an auxiliary particle map with prescribed fields and frozen orbit data. A full determinant theorem for the fully coupled implicit PIC residual is a different and stronger claim that we leave to future work.

\begin{proposition}[Endpoint transformation]
Let $\Ah^n(\xx)$ be prescribed.  The map
\begin{equation}
  T^n(\xx,\PP)=(\xx,\pp),\qquad \pp=\PP-q\Ah^n(\xx),
\end{equation}
preserves $\dd\xx\dd\PP$.
\end{proposition}
\begin{proof}
The Jacobian is block lower triangular:
\begin{equation}
  DT^n=\begin{pmatrix}
  I&0\\ -q\grad\Ah^n(\xx)&I
  \end{pmatrix}.
\end{equation}
The determinant is the product of the diagonal block determinants and is therefore one.
\end{proof}

\begin{theorem}[Fixed-field GM--HC particle volume]
\label{thm:fixed-volume}
Assume that $\Ah^n$, $\Ah^{n+1}$, $\gphibar_i$, $\DA_i\AV$, $\KA_i\AV$, and $\vbar_i$ are prescribed data for one step.  If the magnetic substep in \cref{eq:gmhc-rot} is the HC rotation generated by $\KA_i\AV$ as defined in Definition~\ref{def:hc-skew}, then the auxiliary one-particle GM–HC map with frozen orbit data preserves canonical particle volume:
\begin{equation}
  \det\frac{\partial(\xx_i^{n+1},\PP_i^{n+1})}{\partial(\xx_i^n,\PP_i^n)}=1.
\end{equation}
Consequently the product map preserves $\prod_i\dd \xx_i\dd \PP_i$ for independent particles in the same prescribed fields.
\end{theorem}
\begin{proof}
With frozen orbit data, \cref{eq:gmhc-pre,eq:gmhc-post} are momentum translations by functions of the endpoints and prescribed field data.  These translations are block triangular with unit diagonal blocks, just as in the endpoint transformation above.  The magnetic substep is an ordinary HC magnetic rotation with fixed effective magnetic field $B_{\rm eff}(\KA_i\AV)$ and therefore has determinant one.  The position drift written in the usual centered HC variables is a shear.  The determinant of the composition is the product of the determinants of these factors, hence one.
\end{proof}

\begin{remark}[What is not claimed]
In the fully coupled orbit-averaged PIC method, $\Sbar_{ig}$, $\DA_i\AV$, $\KA_i\AV$, the fields, and the accepted orbit all depend on the nonlinear solution.  A full-system volume theorem would require analyzing the residual map $R(z^{n+1},z^n)=0$ and proving $\det(-R_{z^{n+1}}^{-1}R_{z^n})=1$.  Theorem~\ref{thm:fixed-volume} is the fixed-field particle-volume result that transfers the HC structure into generalized momentum.  The energy theorem in \cref{sec:energy} is the fully coupled conservation statement established for the converged GM--HC--CN fixed point.
\end{remark}
\section{Nondimensional relativistic energy balance for the orbit-averaged GM--HC--CN fixed point}
\label{sec:energy}

The fixed-field theorem in \cref{sec:gmhc-volume} explains why the accepted particle push retains the HC volume-preserving magnetic rotation.  The fully coupled PIC theorem has a different target: exact conservation of the nondimensional relativistic total energy for the converged nonlinear step.    The key point is that the orbit average used by the particle push does not replace the HC magnetic rotation.  It supplies the current, gather, and vector-potential chain rule required by the same energy argument developed in Part I, now with the relativistic kinetic-energy secant velocity.

\subsection{Nondimensional relativistic kinetic-energy secant velocity}\label{sec:rel-vel}

For particle \(i\), define the nondimensional relativistic kinetic energy by
\begin{equation}
  \Kkin_i(\pp)=m_i\kappaC^2\bigl(\gamma_i(\pp)-1\bigr),
  \qquad
  \gamma_i(\pp)=\sqrt{1+\frac{|\pp|^2}{m_i^2\kappaC^2}}.
  \label{eq:Kkin-nondim}
\end{equation}
This is the particle kinetic energy associated with the Hamiltonian in \cref{eq:rel-Hamiltonian}.  For two mechanical momenta \(\pp_a,\pp_b\in\mathbb R^3\), set
\begin{equation}
  \VV_i^{\rm dg}(\pp_a,\pp_b)=
  \frac{\pp_a+\pp_b}{m_i\bigl(\gamma_i(\pp_a)+\gamma_i(\pp_b)\bigr)},
  \label{eq:Vdg}
\end{equation}
where we have used the superscript $\text{dg}$ to denote the discrete-gradient velocity associated with the relativistic kinetic energy between two momentum states $\pp_b$ and 	$\pp_a$.
Then
\begin{equation}
  \Kkin_i(\pp_a)-\Kkin_i(\pp_b)
  =\VV_i^{\rm dg}(\pp_a,\pp_b)\cdot(\pp_a-\pp_b).
  \label{eq:rel-sec}
\end{equation}
Indeed, \(\gamma_i(\pp_a)-\gamma_i(\pp_b)\) factors the difference of squares \(|\pp_a|^2-|\pp_b|^2\), and the factors of \(\kappaC\) cancel exactly as in \cref{eq:rel-particle-def}.  Thus \(V_i^{\rm dg}\) is a discrete-gradient, or secant, velocity for the nondimensional relativistic kinetic energy.  The orbit/current velocity is constructed from the two electric-work secant velocities as specified in \cref{eq:vbar-energy}.

\subsection{Mechanical split and magnetic zero work}\label{sec:Mechanical-split}

Let the endpoint mechanical momenta be
\begin{equation}
  \pp_i^n=\PP_i^n-q_i\Ah^n(\xx_i^n),
  \qquad
  \pp_i^{n+1}=\PP_i^{n+1}-q_i\Ah^{n+1}(\xx_i^{n+1}).
  \label{eq:mechanical-endpoints-energy}
\end{equation}
Define the midpoint mesh electric field and its orbit gather by
\begin{equation}
  \EE_g^{\half}=-(\gradh\phi^{\half})_g-\UU_g^{\half},
  \qquad
  \Ebar_i=\sum_g \EE_g^{\half}\Sbar_{ig}
  =-\gphibar_i-\Ubar_i,
  \label{eq:Ebar-energy}
\end{equation}
where
\begin{equation}
  \gphibar_i=\sum_g(\gradh\phi^{\half})_g\Sbar_{ig},
  \qquad
  \Ubar_i=\sum_g\UU_g^{\half}\Sbar_{ig}.
\end{equation}
Using the orbit-gradient split from \cref{eq:D-split}, the mechanical form equivalent to the canonical no-explicit-\(\partial_t\AV\) update is
\begin{subequations}
\label{eq:mechanical-split}
\begin{align}
  \pp_i^- &= \pp_i^n+\frac{q_i\dt}{2}\Ebar_i,
  \label{eq:ehalf1}\\
  \pp_i^+-\pp_i^- &= q_i\dt(\KA_i\AV)
  V_i^{\rm HC}(\pp_i^+,\pp_i^-;\KA_i\AV),
  \label{eq:magrot-energy}\\
  \pp_i^{n+1} &= \pp_i^+ +\frac{q_i\dt}{2}\Ebar_i.
  \label{eq:ehalf2}
\end{align}
\end{subequations}
Here \textcolor{black}{$\KA_i\AV$ is defined by \eqref{eq:D-split-skew},}
% \begin{equation}
%   \KA_i\AV:=(\DA_i\AV)^T-\DA_i\AV,
%   \qquad
%   (\KA_i\AV)^T=-(\KA_i\AV),
%   \label{eq:Ksplit-energy}
% \end{equation}
and the HC magnetic velocity is the species-\(i\) specialization of \cref{eq:VHC-def}.  Equivalently, define
\begin{equation}
  \uu_i^-:=\frac{\pp_i^-}{m_i\kappaC},
  \qquad
  \Gamma_i^{\AV}:=
  \gamma_{\rm HC}\!\left(\uu_i^-;\BB_{{\rm eff},i}(\AV)\right),
  \label{eq:GammaA-energy}
\end{equation}
where \(\gamma_{\rm HC}\) is the nondimensional HC scalar from \cref{eq:gammaHC}, evaluated with the species parameters \(q_i\), \(m_i\), and the same time step \(\dt\).  Then
\begin{equation}
  V_i^{\rm HC}(\pp_i^+,\pp_i^-;\KA_i\AV)
  =\frac{\pp_i^+ + \pp_i^-}{2m_i\Gamma_i^{\AV}},
  \qquad
  \pp_i^+=\bR_{\rm HC}^p(\pp_i^-;\KA_i\AV).
  \label{eq:VHC-energy-def}
\end{equation}
Thus the magnetic increment in \cref{eq:magrot-energy} is exactly the Cayley/HC relation \cref{eq:hc-magnetic-increment} generated by the orbit skew matrix \(\KA_i\AV\).  Since \(V_i^{\rm HC}\) is parallel to \(\pp_i^++\pp_i^-\) and \(\KA_i\AV\) is skew,
\begin{equation}
  (\pp_i^++\pp_i^-)\cdot
  (\KA_i\AV)V_i^{\rm HC}(\pp_i^+,\pp_i^-;\KA_i\AV)=0.
  \label{eq:magnetic-zero-work-energy}
\end{equation}
Taking the dot product of \cref{eq:magrot-energy} with \(\pp_i^++\pp_i^-\) therefore gives
\begin{equation}
  |\pp_i^+|^2-|\pp_i^-|^2=0,
  \qquad
  \Kkin_i(\pp_i^+)-\Kkin_i(\pp_i^-)=0.
  \label{eq:magnetic-kinetic-zero}
\end{equation}
The HC substep supplies the relativistic magnetic rotation, while the orbit skew matrix supplies the magnetic generator in potential form.

Define the electric-work velocities
\begin{equation}
  \vv_{i}^{E,-}=\VV_i^{\rm dg}(\pp_i^-,\pp_i^n),
  \qquad
  \vv_{i}^{E,+}=\VV_i^{\rm dg}(\pp_i^{n+1},\pp_i^+),
  \label{eq:electric-work-velocities}
\end{equation}
and the orbit/current velocity
\begin{equation}
  \vbar_i=\frac{1}{2}\left(\vv_i^{E,-}+\vv_i^{E,+}\right).
  \label{eq:vbar-energy}
\end{equation}
The accepted orbit is \(\xx_i^{n+1}=\xx_i^n+\dt\vbar_i\), and the current is deposited with this same velocity and the same orbit weights:
\begin{equation}
  \JJ_g^{\half}=\frac{1}{\dV}\sum_i q_i\vbar_i\Sbar_{ig}.
  \label{eq:orbit-current-energy}
\end{equation}

Using \cref{eq:rel-sec} on the two electric half-kicks, using \cref{eq:magnetic-kinetic-zero} for the HC magnetic rotation, and then using \cref{eq:vbar-energy}, we obtain the single-particle balance
\begin{equation}
  \Kkin_i(\pp_i^{n+1})-\Kkin_i(\pp_i^n)
  =q_i\dt\Ebar_i\cdot\vbar_i.
  \label{eq:single-particle-work}
\end{equation}
Summing over particles and applying the orbit-averaged scatter/gather identity \cref{eq:scatter-gather-work} with \(\bm{G}=\EE^{\half}\) gives
\begin{equation}
  \Kkin^{n+1}-\Kkin^n
  =\dt\sum_g \JJ_g^{\half}\cdot \EE_g^{\half}\dV,
  \label{eq:particle-work-balance}
\end{equation}
where $\Kkin^n=\sum_i m_i\kappaC^2\bigl(\gamma_i(\pp_i^n)-1\bigr).$
% \begin{equation}
  
%   \label{eq:total-kinetic-energy}
% \end{equation}

\subsection{CN field energy and total energy theorem}

The endpoint and midpoint mesh fields are
\begin{equation}
  \EE^n=-\gradh\phi^n-\UU^n,
  \qquad
  \BB^n=\curlh\AV^n,
  \qquad
  \EE^{\half}=-\gradh\phi^{\half}-\UU^{\half}.
  \label{eq:energy-fields}
\end{equation}
Using the nondimensional field energy $W^n$ defined in \cref{eq:field-energy-section2}, the CN potential solve \cref{eq:CN-field}, the midpoint Lorenz gauge \cref{eqn:gauge conditions}, the compatible coefficient scaling of \cref{eq:Maxwell_scaling}, and periodic summation by parts give the midpoint Maxwell energy balance
\begin{equation}
  \Wh^{n+1}-\Wh^n
  =-\dt\sum_g\JJ_g^{\half}\cdot \EE_g^{\half}\dV.
  \label{eq:field-work-balance}
\end{equation}
Combining \cref{eq:particle-work-balance} and \cref{eq:field-work-balance} gives the fully discrete result.

\begin{theorem}[Nondimensional relativistic GM--HC--CN energy conservation]
\label{thm:energy}
Assume periodic boundary conditions and summation-by-parts spatial operators. Assume also the midpoint Lorenz gauge \eqref{eqn:gauge conditions}
and the compatible coefficient scaling \eqref{eq:Maxwell_scaling}. Suppose the CN potential equations \cref{eq:CN-field}, the continuity update \cref{eq:cn-continuity}, the orbit current \cref{eq:orbit-current-energy}, the orbit weights, the vector-potential chain rule \cref{eq:A-chain-rule}, the skew magnetic matrix \eqref{eq:D-split-skew}, and the mechanical split \cref{eq:mechanical-split} together with the orbit velocity \cref{eq:vbar-energy} all hold at the accepted step.  If the coupled nonlinear system is solved to convergence, then
\begin{equation}
  \Etot^{n+1}=\Etot^n,
  \qquad
  \Etot^n=\Kkin^n+\Wh^n.
  \label{eq:total-energy-nondim}
\end{equation}
In floating-point arithmetic, the observed defect is limited by nonlinear solver tolerance, orbit-quadrature error, chain-rule residual, and roundoff.
\end{theorem}
\begin{proof}
Equation~\eqref{eq:particle-work-balance} is the particle kinetic-energy change, and \cref{eq:field-work-balance} is the negative field-energy change.  Adding them cancels the mesh work exactly.
\end{proof}

\begin{remark}[Role of orbit averaging]
The orbit average does not replace the HC magnetic rotation by an additive force.  It supplies the frozen data \(\Sbar_{ig}\), \(\DA_i\AV\), and \(\KA_i\AV\) used in the rotation and in the chain rule.  For fixed fields and frozen orbit data, the HC volume-preservation argument of \cref{thm:fixed-volume} still applies.  For the coupled energy theorem, the essential requirement is different: the same \(\vbar_i\) and the same \(\Sbar_{ig}\) must define the orbit, the current deposit, the electric-field gather, and the particle work.
\end{remark}

\subsection{Diagnostics required by the proof}

The implementation reports the total relative energy error and the mesh RMS
norms of the Gauss and Lorenz-gauge residuals
\begin{equation}
  \begin{aligned}
  R_G^n&=\divgh\EE^n-\sigma_1\rho^n,
  \qquad R_L^n=\frac{\psi^n}{\kappaC^2}+\divgh\AV^n,\\
  {\|R_\alpha^n\|_{\rm rms}}
  &={\left(\frac{1}{N_g}\sum_g
  \left|(R_\alpha^n)_g\right|^2\right)^{1/2},
  \qquad \alpha\in\{G,L\}.}
  \end{aligned}
  \label{eq:constraint-residuals}
\end{equation}
where \(N_g\) is the number of mesh points. The implementation also monitors the orbit chain-rule residual
\begin{equation}
  \begin{aligned}
  \bR_{A,i}
  &=\Ah^{n+1}(\xx_i^{n+1})-\Ah^n(\xx_i^n)
  -\dt\left(\Ubar_i+(\DA_i\AV)\vbar_i\right),\\
  \|\bR_A\|_{\rm rms}
  &=\left(
  \frac{1}{3N_p}\sum_{i=1}^{N_p}\|\bR_{A,i}\|_2^2
  \right)^{1/2}.
  \end{aligned}
  \label{eq:RA-diagnostic}
\end{equation}
Finally, it records the per-step particle--grid work residual
\begin{equation}
  R_{\rm work}=\dt\left[
  \sum_i q_i\vbar_i\cdot\Ebar_i
  -\sum_g \JJ_g^{\half}\cdot \EE_g^{\half}\dV
  \right].
  \label{eq:work-diagnostic}
\end{equation}
If the orbit chain-rule residual \(\|\bR_A\|_{\rm rms}\) is not small, the orbit splitting, quadrature, and discrete-gradient construction should be checked.  If \(R_{\rm work}\) is not small, scatter and gather are not adjoint.  If both are small but total energy drifts, the CN field-energy identity, coefficient scaling, gauge relation, or nonlinear convergence must be checked.

\subsection{Implicit energy-conserving PIC solver}
\label{subsec:adaptive_particle_step}
The fully implicit PIC solver presented in this work builds on the nonrelativistic method in \cite[Algorithm~4.1]{ChristliebChaconGong2026EC} with the following four changes. First, the GM--HC map defined in
\cref{sec:rel-vel,sec:Mechanical-split} replaces the nonrelativistic
particle update. Second, the particle-shape degree is raised from \(r=1\) to
\(r=2\). Third, an orthogonal Fourier projector restricts the fields and source current to the subspace excluding even-grid Nyquist planes \cite{GottliebOrszag1977}. Fourth, the coupled nonlinear system is solved by an outer Anderson-accelerated current iteration with inner Picard particle iterations.
The underlying orbit construction, knot-split quadrature, continuity update, CN potential solve, and adjoint scatter–gather structure are retained. The CN potential equations are solved directly, mode by mode in Fourier space, on the periodic spectral mesh. This linear field update is included in each evaluation of the coupled nonlinear map described below.

We first describe the change in particle shape. For
\(\xi=(x-x_g)/\Delta x\), the centered one-dimensional quadratic
B-spline used here is
\begin{equation*}
S_2(\xi)
=\frac{1}{2}\left[
\left[\left(\frac{3}{2}-|\xi|\right)_+\right]^2
-3\left[\left(\frac{1}{2}-|\xi|\right)_+\right]^2
\right],
\qquad x_+ := \max(x,0).
\end{equation*}
In three dimensions the particle shape is the tensor product
\(S_g(\xx)=\prod_{\ell=1}^3
S_2((x_\ell-x_{g,\ell})/\Delta x_\ell)\). It is therefore piecewise quadratic
in each coordinate, rather than one total-degree-two polynomial in three
variables, and has support on at most \(3^3=27\) mesh points. Its \(C^1\)
regularity satisfies the requirement of the potential-based magnetic generator.
Degree two is a sufficient choice rather than a minimality claim;
\cref{sec:results} gives numerical evidence.

\paragraph{Nyquist projection}
Let $(N_1,N_2,N_3)=(N_x,N_y,N_z)$ denote the numbers of mesh points in the three coordinate directions, so that $N_g=\prod_{\ell=1}^3N_\ell=N_xN_yN_z$ is the total number of mesh points and let $\mathcal P$ be the orthogonal Fourier projector that removes every mode lying on any even-grid Nyquist plane. Here $\boldsymbol k$ denotes the unshifted integer FFT index, with $k_\ell\in\{0,\ldots,N_\ell-1\}$ for $\ell\in\{1,2,3\}$:
\begin{equation}
 \widehat{\mathcal P f}_{\boldsymbol k}=
 \begin{cases}
 0,&k_\ell=N_\ell/2\text{ for at least one even }N_\ell,\\
 \widehat f_{\boldsymbol k},&\text{otherwise}.
 \end{cases}
 \label{eq:nyquist-projector}
\end{equation}
The projection is applied componentwise to vector fields. It restricts the field space while leaving all retained Fourier coefficients unchanged. Initial charge, initial current, the source current at every nonlinear evaluation, potentials, and potential time derivatives are projected. Canonical particle momenta are reconstructed from the prescribed mechanical velocities and projected vector potential. The continuity and CN updates remain in the retained field space.

The projection addresses a distinction between the mesh derivative and the derivative of the particle interpolant at the Nyquist frequency. In our implementation, the first-derivative symbol vanishes at the even-grid Nyquist index, and the mesh Laplacian is constructed from these same symbols. Thus the mesh operators remain mutually compatible even without projection. However, a Nyquist component of the vector potential can have a nonzero derivative after spline interpolation to particle positions.

For example, consider $A_{x,g}=a(-1)^{j_y}$, where $j_y$ is the mesh index in the $y$ direction, with all other components zero and no dependence on $x$ or $z$. Its represented mesh curl vanishes, whereas $\partial_y A_{h,x}$ generally does not vanish between mesh nodes. Consequently, this component can contribute to the orbit skew matrix $\KA_i\AV$ and rotate particle momentum despite contributing no mesh magnetic field. Such a rotation performs no particle work, so energy conservation alone does not exclude this effect.

Removing the Nyquist planes eliminates these potential components from both the field solve and the particle interpolant. The restriction therefore removes this particular source of disagreement between the mesh and particle magnetic representations. It does not make spline differentiation identical to gathered spectral differentiation for all retained modes.

The projected current is $\bm J=\mathcal P\bm J_{\rm orb}$. Since $\mathcal P$ is self-adjoint and the mesh electric field satisfies \(\mathcal P\mathbf E=\mathbf E\),
\begin{equation}
 \langle\mathcal P\bm J_{\rm orb},\EE\rangle_h
 =\langle\bm J_{\rm orb},\mathcal P\EE\rangle_h
 =\langle\bm J_{\rm orb},\EE\rangle_h.
 \label{eq:projected-work-identity}
\end{equation}
Thus the deposit--gather work cancellation holds on the retained field space. The orbit chain rule uses the same projected interpolant at its endpoints and along the orbit. The projector commutes with the spectral derivatives, so the continuity and field-energy arguments apply in this subspace. This is a restriction of the coupled discretization, rather than a correction to particle velocities after a time step.

\paragraph{\texorpdfstring{Current-based formulation with field and particle elimination}{Current-based formulation with field and particle elimination}}
The nonlinear iteration is formulated entirely in terms of the mesh current. As summarized in \cref{alg:anderson-orbit-cn}, each residual evaluation first projects the trial current, advances charge through continuity, and solves the linear CN potential equations directly in Fourier space. With these trial fields fixed, inner Picard iterations determine the particle orbits. Deposition and projection of the resulting orbit current define the map $H(\bm j)$, and the outer iteration enforces $H(\bm j)-\bm j=0$. Both the field variables and the particle orbit variables are therefore eliminated from the outer nonlinear system. The direct Fourier solve incorporates the linear electromagnetic response, including the light-wave dynamics, within each residual evaluation. The remaining current consistency problem contains the coupled plasma response; no additional physics-based preconditioner for this reduced problem is used here.

Let $\bm j=\operatorname{vec}
    \left(\{\bm J_g^{n+1/2,\mathrm{trial}}\}_{g=1}^{N_g}\right) \in\mathbb R^{3N_g}$ denote the stacked trial
mesh-current vector.
The field update and converged inner particle
solves determine the orbit current $\bm J_{\rm orb}(\bm j)$.
Define $H(\bm j)=\operatorname{vec}
    \left(\mathcal P\bm J_{\rm orb}(\bm j)\right).$
The outer problem is $H(\bm j)=\bm j$, with residual
$\bm r=H(\bm j)-\bm j$ and mesh RMS norm 
\[
  \|\bm r\|_{\rm rms}
    =\left(\frac{1}{3N_g}\sum_{g=1}^{N_g}|\bm r_g|^2\right)^{1/2}.
\] 
% \[
%   \|\bm r\|_{\rm rms}
%     =\left(1/(3N_g)\sum_{g=1}^{N_g}|\bm r_g|^2\right)^{1/2}.
% \] 
The same RMS norm is used for the input and output
mesh-current vectors in the convergence criteria below. For the same discrete equations, projection, and orbit quadrature, convergence of the inner and outer iterations recovers a solution of the coupled particle–field system. Particle enslavement changes the nonlinear unknowns, rather than the discrete equations.
The particle trajectories are functions of the trial current through these inner solves; they are not entries of the outer Anderson vector.
The outer nonlinear vector contains $3N_g$ mesh-current components,
rather than $3N_p$ particle orbit-velocity components. For Anderson
acceleration with the same history length $M$, this reduces the
history storage from $O(MN_p)$ to $O(MN_g)$. Particle states and
orbit data still require $O(N_p)$ storage. For example, on the $32^3$ meshes used below, the outer vector contains $98{,}304$ components, compared with $1{,}966{,}080$ particle-velocity components for TSI and $4{,}915{,}200$ for WBI. This addresses the nonlinear-vector memory concern without claiming that total solver memory is independent of particle number. The inner iterations converge particle orbits over a fixed field step; the four main projected runs below use one particle time interval per field step. Particle enslavement also provides a basis for future large-scale parallel implementations with particle subcycling, with field macro time-step fixed and particle time-step refinement and consistently accumulated orbit-averaged currents, following \cite{ChenChaconBarnes2011,ChenChacon2015}. Large-scale parallel implementation and accuracy-controlled particle subcycling are beyond the present study.

The outer acceptance tests are
% \begin{subequations}\label{eq:mixed-nonlinear-tests}
% \begin{align}
%  \frac{\|\bm r\|_{\rm rms}}
%  {\tau_{\rm abs}+\frac{\tau_{\rm rel}}{2}
%  (\|H(\bm j)\|_{\rm rms}+\|\bm j\|_{\rm rms})}&\leq1,
%  \label{eq:mixed-nonlinear-rms}\\
%  \max_g\frac{|\bm r_g|}
%  {\tau_{\rm abs}+\frac{\tau_{\rm rel}}{2}
%  (|H_g(\bm j)|+|\bm j_g|)}&\leq1.
%  \label{eq:mixed-nonlinear-max}
% \end{align}
% \end{subequations}
\begin{align}
\label{eq:mixed-nonlinear-tests}
 \frac{\|\bm r\|_{\rm rms}}
 {\tau_{\rm abs}+\frac{\tau_{\rm rel}}{2}
 (\|H(\bm j)\|_{\rm rms}+\|\bm j\|_{\rm rms})} \leq1, \quad 
 \max_g\frac{|\bm r_g|}
 {\tau_{\rm abs}+\frac{\tau_{\rm rel}}{2}
 (|H_g(\bm j)|+|\bm j_g|)} \leq1.
\end{align}
The maximum is over mesh nodes. Independently, each inner particle solve uses a mixed absolute--relative test on its orbit-velocity residual and requires finite, subluminal velocities. The inner solver is damped Picard; the projected production runs use unit inner damping. Repeated inner iterations converge an orbit over the same full time interval and do not constitute time subcycling.

The outer solver uses limited-memory type-II Anderson acceleration. With differences $\Delta\bm j$ and $\Delta\bm r$ stored as columns of $\Delta\mathcal J$ and $\Delta\mathcal R\in\mathbb R^{3N_g\times m_h}$, $m_h\leq M$, it computes
\begin{equation}
 \left[\Delta\mathcal R^T\Delta\mathcal R+
 \lambda\frac{\operatorname{tr}(\Delta\mathcal R^T\Delta\mathcal R)}{m_h}I\right]\bm\gamma
 =\Delta\mathcal R^T\bm r, \qquad
 \bm j_{\rm A}=\bm j+\beta\bm r-
 (\Delta\mathcal J+\beta\Delta\mathcal R)\bm\gamma.
 \label{eq:anderson-type-two-update}
\end{equation}
The history is cleared when unusable or when the relative RMS residual grows by more than a factor $1.2$. Backtracking tests $\alpha=1,1/2,\ldots,1/64$ along the proposed direction. An accepted trial must reduce the scaled RMS residual by a factor at most $1-10^{-4}\alpha$ and keep the scaled maximum-node residual within a factor $1.2$ of its current value. A rejected Anderson direction triggers a damped-Picard attempt. Failure leaves the accepted state unchanged.

\begin{algorithm}[!htbp]
\caption{Current-based GM--HC--CN solve with field and particle elimination and Nyquist projection}
\label{alg:anderson-orbit-cn}
\footnotesize
\begin{algorithmic}[1]
\Require Accepted particle--field state, fixed $\dt$, outer and inner tolerances, and iteration limits.
\Ensure A converged time-$n+1$ state, or failure with the accepted state unchanged.
\State Initialize the trial current from the accepted current and clear the outer history.
\For{each outer iteration}
 \State Project the trial current, advance charge by continuity, and solve the CN potential equations.
 \State Hold the trial fields fixed and converge every particle orbit by inner Picard iteration over the full $\dt$.
 \If{any inner particle solve fails}
  \State Return a failed-step flag.
 \EndIf
 \State Deposit and project the orbit current to obtain $H(\bm j)$; form $\bm r=H(\bm j)-\bm j$.
 \If{both outer tests in \cref{eq:mixed-nonlinear-tests} pass}
  \State Re-evaluate the full map at the candidate current and verify both outer tests and all inner solves.
  \State Commit the corresponding particle, source, and field state if verified; otherwise return failure.
  \State \Return
 \EndIf
 \State Form a safeguarded Anderson update, or a damped-Picard update if the history is unusable.
 \State Backtrack using the RMS-decrease and maximum-node safeguards; retry a rejected Anderson direction with damped Picard.
 \If{neither direction is accepted}
  \State Return a failed-step flag.
 \EndIf
\EndFor
\State Return failure if the outer iteration limit is reached.
\end{algorithmic}
\end{algorithm}

\section{Numerical results}
\label{sec:results}
\label{sec:projected-results}
The numerical implementation is written in C++17 and parallelized with shared-memory OpenMP. We consider two representative examples:
the cold relativistic two-stream instability (TSI) and the cold
relativistic Weibel instability (WBI). 
The source code and input files used for the numerical experiments are publicly available in the GitHub repository~\cite{Gong2026UnstaggeredPIC}. 
All simulations use the
nondimensional normalization
$
  \kappa=1,\,
  \sigma_1=\sigma_2=1,\,
  \rho_0=1,
$
with electron macroparticles satisfying \(q/m=-1\) and a uniform,
immobile neutralizing background. This choice satisfies
\(\sigma_1=\kappa^2\sigma_2\) and gives the cold electron plasma frequency
\(\omega_p^2=\sigma_1\rho_0=1\). Each macroparticle has
\(m_i=|q_i|\), and all particle speeds satisfy \(|\vv_i|<\kappa\).

The four main production runs use \(v_0=0.9\), hence
\(\gamma_0=(1-v_0^2/\kappa^2)^{-1/2}=2.2941573387\), a \(32^3\) mesh, fixed
time steps, quadratic tensor-product B-splines (\(r=2\)), spectral fields and Nyquist projection. Gauss--Legendre quadrature of order eight is applied separately on every knot-split
orbit interval. \textcolor{black}{The absolute outer tolerance is
\(\tau_{\rm abs}=10^{-13}\), and the relative outer tolerance is \(\tau_{\rm rel}=10^{-11}\). The inner limit is 64 iterations per particle solve, and the outer limit is 80 iterations per field step.} The Anderson solver uses memory \(M=4\) and becomes eligible at iteration \(m_{\rm A}=2\); it employs damping \(\beta=0.5\), regularization \(\lambda=10^{-12}\). Its line search tests
\(\alpha=1,1/2,\ldots,1/64\), with sufficient-decrease coefficient
\(10^{-4}\) and residual-growth limit \(1.2\). Each particle pair
comprises two electron macroparticles with symmetry-related initial
velocities. The parameters that differ among the four simulations are summarized
in \Cref{tab:simulation-parameters}. For each instability, we compare a smaller and a larger fixed time step to assess sensitivity of the reported diagnostics to the field time step. The selected values give an integer number of steps over each simulation interval. For the axis-aligned modes, the B-spline Fourier shape factor and finite-grid plasma frequency are
\begin{equation}
 \chi_r(k)=\left[\operatorname{sinc}\left(\frac{k\Delta x}{2}\right)\right]^{r+1},
 \qquad \operatorname{sinc}(z)=\frac{\sin z}{z},
 \qquad \omega_{p,h}^2=\sigma_1\rho_0|\chi_r(k)|^2.
 \label{eq:shape-factor-results}
\end{equation}
The growth-rate reference includes this spatial shape factor. Plotted mode amplitudes are $|\widehat f_{\boldsymbol k}|/N_g$, where $\widehat f$ is the unnormalized discrete Fourier transform of the mesh field. Thus a sinusoid of physical amplitude $B_0$ has positive-mode amplitude $B_0/2$. 
We monitor the unstable-mode
amplitude, relative total-energy change, Gauss and Lorenz residuals, orbit
chain-rule residual \(R_A\) from \cref{eq:RA-diagnostic}, and deposit--gather
work residual \(R_{\rm work}\) from \cref{eq:work-diagnostic}. Growth rates are obtained by least-squares fits of log mode amplitude against time over \(0\leq t\leq2\) for TSI and \(0\leq t\leq1.6\) for WBI. Every step in the four runs in \Cref{tab:simulation-parameters} passed the final mixed residual tests.
\begin{table}[!htbp]
\color{black}\centering\small
\caption{Reproducibility parameters for the four main simulations. Common Anderson and initialization settings are stated in the text.}
\label{tab:simulation-parameters}
\setlength{\tabcolsep}{4pt}
\renewcommand{\arraystretch}{1.10}
\begin{tabular}{lrrrr}
\hline
Parameter & TSI-0.1 & TSI-0.3 & WBI-0.04 & WBI-0.2 \\
\hline
Domain & $[-4\pi,4\pi]^3$ & $[-4\pi,4\pi]^3$ & $[-\pi,\pi]^3$ & $[-\pi,\pi]^3$ \\
Pairs per cell & 10 & 10 & 25 & 25 \\
Unstable wave number & $k_x=1/4$ & $k_x=1/4$ & $k_y=1$ & $k_y=1$ \\
$\Delta t$ & 0.1 & 0.3 & 0.04 & 0.2 \\
Analysis end time & 60 & 60 & 100 & 100 \\
Inner relative tolerance & $10^{-11}$ & $10^{-11}$ & $10^{-12}$ & $10^{-12}$ \\
Inner absolute tolerance & $10^{-13}$ & $10^{-13}$ & $10^{-14}$ & $10^{-14}$ \\
\hline
\end{tabular}
\end{table}

\subsection{Cold relativistic two-stream instability with
\texorpdfstring{\(v_0=0.9\)}{v0 = 0.9}}
\label{sec:two-stream-results}
The TSI calculation is three dimensional, although its longitudinal growing
mode depends only on \(x\). The mobile charge consists of two equal cold
electron streams, each with unperturbed density \(\rho_0/2\), drifting at
\(\pm v_0\hat x\). The listed domain and mesh give
\(\Delta x=\pi/4\) and \(k_x\Delta x/2=\pi/32\).
The long box resolves the relativistically narrowed unstable band: the
longitudinal response scales with \(\gamma_0^{-3}\), and
a box of side length \(2\pi\), whose fundamental mode has
\(k_x=1\), would be stable for these parameters. 

The initial condition is the growing cold-fluid eigenmode rather than a
density-only perturbation. Its grid charge-density and electric-field
amplitudes satisfy
\begin{equation}
  \rho(x,0)=\alpha_{\rm TS}\rho_0\cos(k_xx),
  \;\;
  E_x(x,0)=\frac{\sigma_1\alpha_{\rm TS}\rho_0}{k_x}\sin(k_xx),
  \;\;
  \alpha_{\rm TS}=5\times10^{-3}.
  \label{eq:two-stream-initial-fields}
\end{equation}
The particle velocities and weights are initialized from the corresponding
finite-grid eigenvector using the gathered field amplitude
\(\chi_2(k_x)E_x\). This suppresses the startup transient associated with a
single-component seed and makes \(|E_x(k_x,t)|\) a clean linear diagnostic.
For two equal cold relativistic beams, the longitudinal dispersion
relation is
\cite{KrallTrivelpiece1973,Stix1992,BretGremilletDieckmann2010}
\begin{equation}
  1-\frac{\omega_{p,h}^2}{2\gamma_0^3}
  \left[
    \frac{1}{(\omega-k_xv_0)^2}
    +\frac{1}{(\omega+k_xv_0)^2}
  \right]=0.
  \label{eq:rel-ts-cont}
\end{equation}
With \(a=\omega_{p,h}^2/(2\gamma_0^3)\) and \(b=k_x^2v_0^2\), the unstable
branch is \(\omega=\ii\Gamma_{\rm TS}\), where
\begin{equation}
  \Gamma_{\rm TS}
  =
  \left[
    \sqrt{a^2+4ab}-a-b
  \right]^{1/2}.
  \label{eq:gamma-ts}
\end{equation}
Equation~\eqref{eq:shape-factor-results} gives
\(|\chi_2(k_x)|^2=0.9904049451\) and
\(\Gamma_{\rm TS}=9.109138040\times10^{-2}\). 

The fitted rates are $0.09101901$ and $0.09095955$ for $\Delta t=0.1$ and $0.3$, respectively, with relative errors $7.95\times10^{-4}$ and $1.45\times10^{-3}$ against the finite-grid reference. Both mode histories initially follow exponential growth and subsequently enter nonlinear trapping. Their sampled first large maxima occur near $t=42.6$, with amplitudes $0.25296$ and $0.25141$. At $T=60$, the amplitudes are $0.0405932$ and $0.0400176$, a difference of about $1.42\%$ relative to the smaller-step run. These comparisons show close agreement for these particular mode diagnostics.

The maximum relative total-energy errors are $3.53\times10^{-12}$ and $6.17\times10^{-12}$. The Gauss, Lorenz, chain-rule, and particle--grid work diagnostics remain small in both runs, as shown in \cref{fig:tsi-residuals}. \Cref{fig:diagnostic-comparison} compares energy drift, structural residuals, and mode amplitudes for the two instabilities. \Cref{fig:phase-tsi-03} shows the $\Delta t=0.3$ TSI $x$--$v_x$ portrait at $T=60$, using 20,000 particles selected by a fixed sampling seed. For $\Delta t=0.1$ and $0.3$, respectively, the all-particle, $|q|$-weighted longitudinal RMS velocities are $0.836845$ and $0.836425$, differing by approximately $0.050\%$. This moment comparison supplements the displayed phase portrait but does not measure all fine-scale phase-space differences. The corresponding transverse RMS velocities are $2.29\times10^{-12}$ and $4.60\times10^{-12}$; their behavior is discussed in the transverse-motion comparison below.

\subsection{Cold relativistic Weibel/filamentation instability}
\label{subsec:weibel_filamentation}
The WBI calculation consists of two equal cold electron beams streaming at
\(V_s=s v_0\hat x\), \(s=\pm1\), with a perturbation varying in \(y\).
The unstable mode has magnetic component \(B_z\) and induced current
\(J_x\), corresponding to the cold symmetric filamentation limit
\cite{Weibel1959,BretGremilletDieckmann2010}. The domain and numerical parameters are specified in \cref{tab:simulation-parameters}. The listed domain and mesh give \(\Delta y=\pi/16\) and
\(|\chi_2(k_y)|^2=0.9904049451\).

The initialization contains the complete growing eigenmode rather than a
magnetic seed alone. With \(B_0=10^{-4}\), its mesh fields are
\begin{subequations}
\label{eq:weibel-initial-fields}
\begin{align}
  A_x(y,0)&=\frac{B_0}{k_y}\cos(k_yy),
  &B_z(y,0)&=-\partial_yA_x=B_0\sin(k_yy),\\
  U_x(y,0)&=\Gamma_W A_x(y,0),
  &E_x(y,0)&=-U_x(y,0).
\end{align}
\end{subequations}
For stream \(s=\pm1\), the matching particle velocity and relative weight are
\begin{subequations}
\label{eq:weibel-initial-particles}
\begin{align}
  v_{x,s}(y,0)
  &=s v_0+\frac{\chi_2(k_y)B_0}{k_y\gamma_0^3}\cos(k_yy),\\
  v_{y,s}(y,0)
  &=s\frac{\chi_2(k_y)v_0B_0}{\gamma_0\Gamma_W}\sin(k_yy),\\
  w_s(y)
  &=1-s\frac{\chi_2(k_y)v_0k_yB_0}
  {\gamma_0\Gamma_W^2}\cos(k_yy).
\end{align}
\end{subequations}
The opposite weight perturbations cancel in the total charge to linear order,
while the current and field time derivative select the growing transverse
electromagnetic branch. The seed \(B_0\) sets the initial mode amplitude but
does not enter the linear growth rate.
Linearization gives
\begin{equation}
  \omega^4-
  \left(
    \kappa^2k_y^2+\frac{\omega_{p,h}^2}{\gamma_0^3}
  \right)\omega^2
  -\frac{k_y^2\omega_{p,h}^2v_0^2}{\gamma_0}=0.
  \label{eq:weibel-dispersion-updated}
\end{equation}
For \(\omega=\ii\Gamma_W\), we have 
$
  \Gamma_W^2
  =
  \frac{1}{2}
  \left[
    -A+\sqrt{A^2+4C}
  \right],
  \quad
  A=\kappa^2k_y^2+\frac{\omega_{p,h}^2}{\gamma_0^3},
  \quad
  C=\frac{k_y^2\omega_{p,h}^2v_0^2}{\gamma_0}.
$
Here the \(\gamma_0^{-3}\) term represents the longitudinal electric response
along the beams, whereas the destabilizing magnetic-bunching term scales as
\(\gamma_0^{-1}\).
For the reported parameters, \(A=1.0820244263\), \(C=0.3496830806\), and
\(\Gamma_W=5.103647904\times10^{-1}\). 

The fitted rates are
\(5.10911239\times10^{-1}\) for WBI-0.04 and
\(5.11347526\times10^{-1}\) for WBI-0.2, with relative errors
\(1.07\times10^{-3}\) and \(1.93\times10^{-3}\), respectively.
Both runs follow the finite-grid prediction during the fitted linear
interval and exhibit similar initial saturation followed by bounded
oscillations over the simulated interval, although differences develop
in the later nonlinear evolution, as shown in
\cref{fig:mode-comparison-wbi}.

For WBI-0.04 and WBI-0.2, respectively, the maximum Gauss RMS residuals
through \(T=100\) are \(1.22\times10^{-12}\) and
\(5.12\times10^{-13}\), the maximum Lorenz RMS residuals are
\(1.40\times10^{-13}\) and \(3.95\times10^{-14}\), and the maximum
absolute relative total-energy changes are \(3.74\times10^{-14}\) and
\(4.13\times10^{-13}\), as shown in
\cref{fig:energy-comparison-wbi,fig:wbi-residuals}.
The corresponding maximum chain-rule RMS residuals are
\(6.68\times10^{-16}\) and \(5.57\times10^{-16}\); the maximum absolute
deposit--gather work residuals are \(1.44\times10^{-14}\) and
\(1.31\times10^{-13}\), respectively.

\begin{figure}[!htbp]
\centering
\captionsetup[subfigure]{font=footnotesize,skip=2pt,labelfont={color=black},justification=centering}
\begin{subfigure}[t]{0.49\linewidth}
\centering
\includegraphics[width=\linewidth]{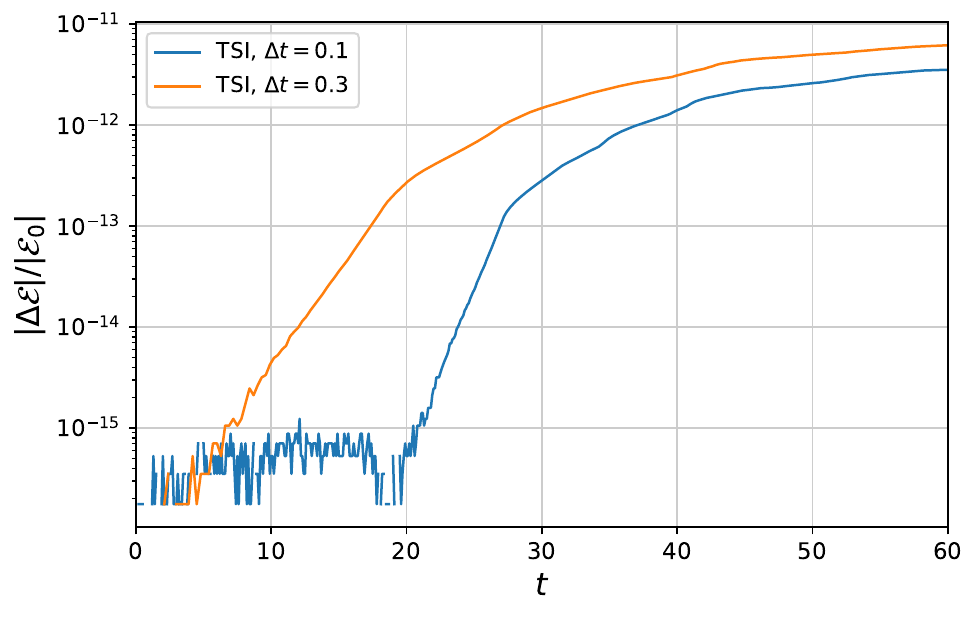}
\caption{TSI energy drift.}
\label{fig:energy-comparison-tsi}
\end{subfigure}\hfill
\begin{subfigure}[t]{0.49\linewidth}
\centering
\includegraphics[width=\linewidth]{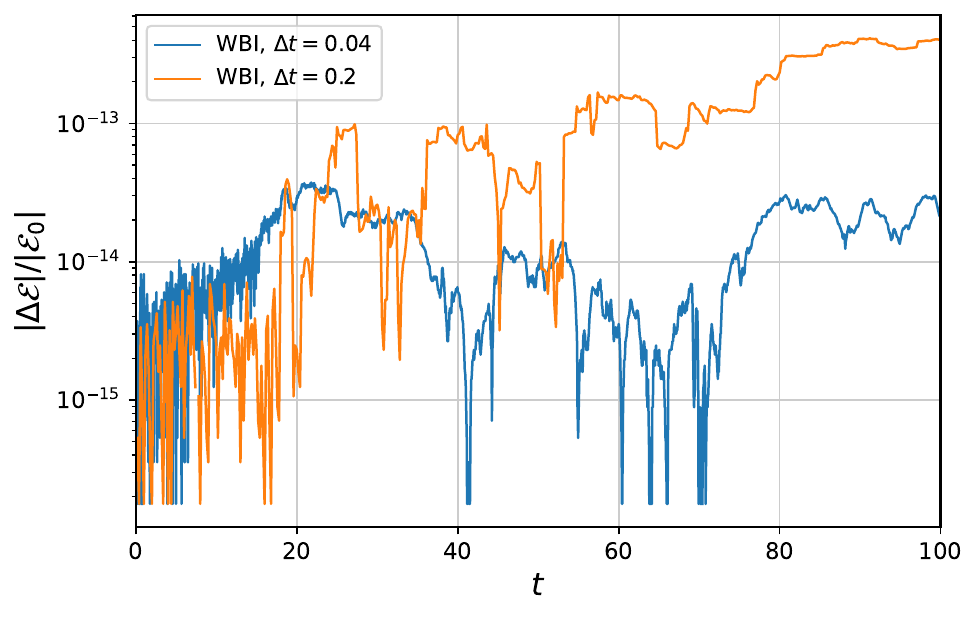}
\caption{Weibel energy drift.}
\label{fig:energy-comparison-wbi}
\end{subfigure}
\par\smallskip
\begin{subfigure}[t]{0.49\linewidth}
\centering
\includegraphics[width=\linewidth]{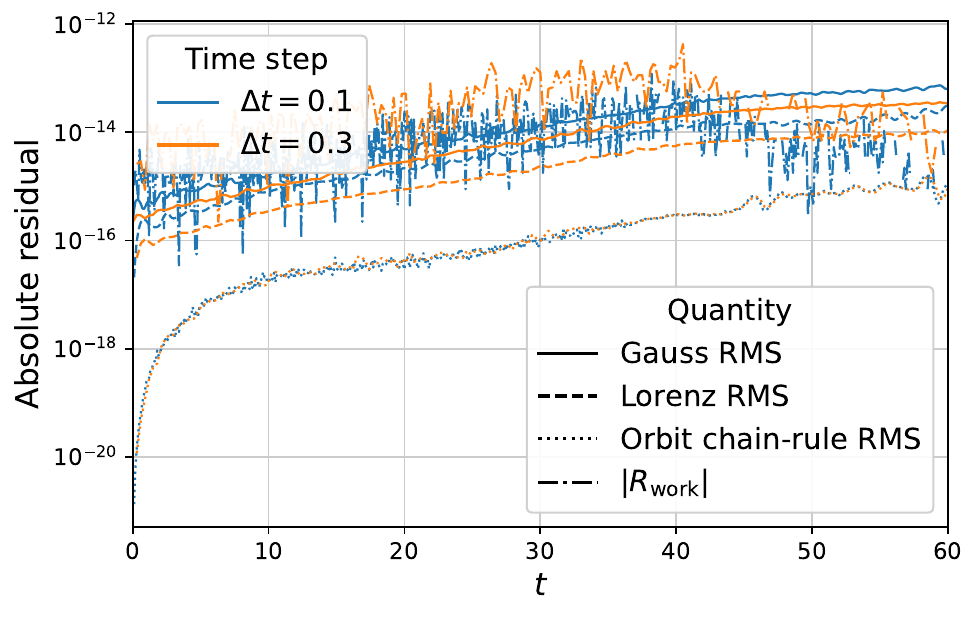}
\caption{TSI structural residuals.}
\label{fig:tsi-residuals}
\end{subfigure}\hfill
\begin{subfigure}[t]{0.49\linewidth}
\centering
\includegraphics[width=\linewidth]{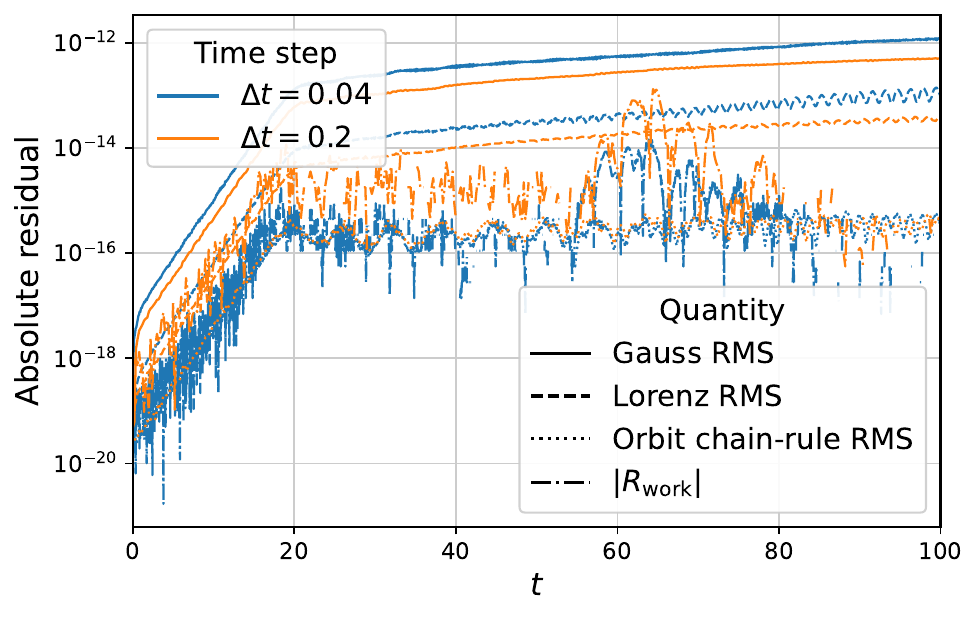}
\caption{Weibel structural residuals.}
\label{fig:wbi-residuals}
\end{subfigure}
\par\smallskip
\begin{subfigure}[t]{0.49\linewidth}
\centering
\includegraphics[width=\linewidth]{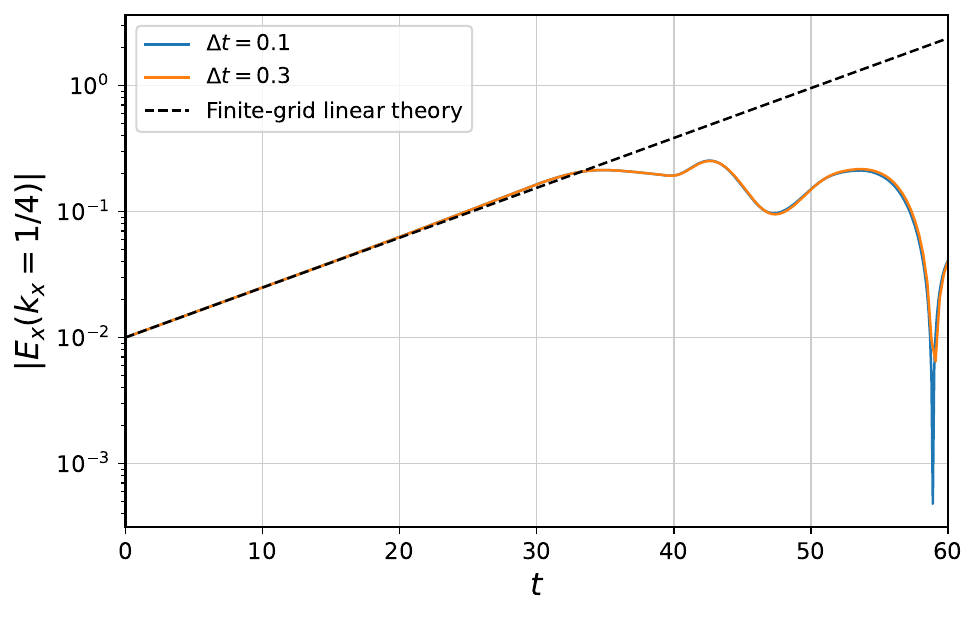}
\caption{TSI electric mode.}
\label{fig:mode-comparison-tsi}
\end{subfigure}\hfill
\begin{subfigure}[t]{0.49\linewidth}
\centering
\includegraphics[width=\linewidth]{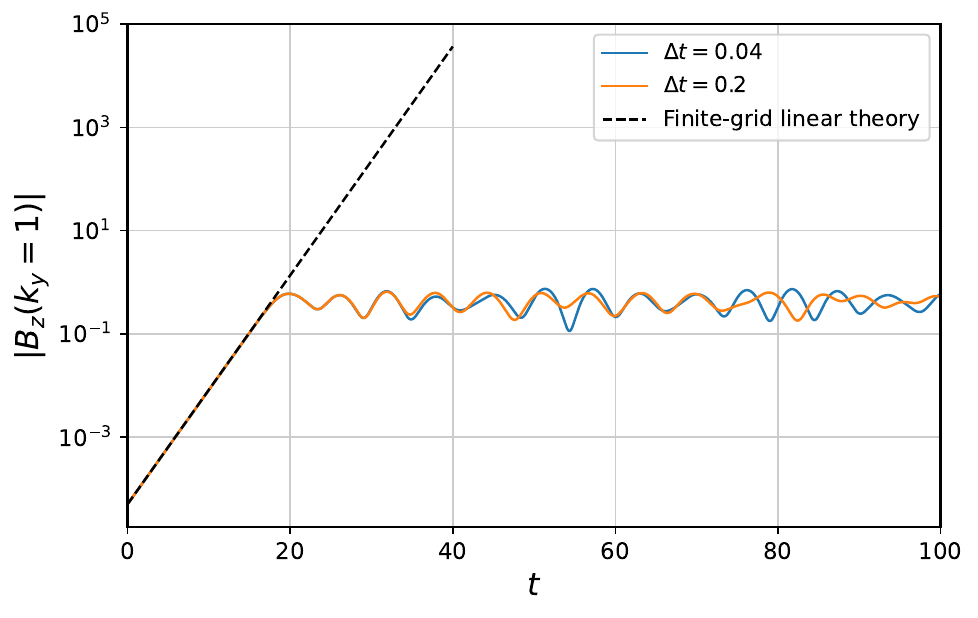}
\caption{Weibel magnetic mode.}
\label{fig:mode-comparison-wbi}
\end{subfigure}
\caption{Projected TSI (left, $\Delta t=0.1,0.3$) and Weibel (right, $\Delta t=0.04,0.2$). Blue/orange denote smaller/larger steps. Top: absolute relative energy drift. Middle: Gauss RMS (solid), Lorenz RMS (dashed), chain-rule RMS (dotted), and absolute deposit--gather work residual (dash-dotted). Bottom: unstable-mode amplitudes through \(t=60\) for TSI and \(t=100\) for Weibel. Growth fits use only $[0,2]$ for TSI and $[0,1.6]$ for Weibel. Exact zeros are omitted on logarithmic axes.}
\label{fig:diagnostic-comparison}
\end{figure}

\subsection{\texorpdfstring{\textcolor{black}{Cross-test solver behavior and summary}}{Cross-test solver behavior and summary}}
\label{sec:results-summary}
The iteration histories in
\cref{fig:iteration-comparison-tsi,fig:iteration-comparison-wbi}
quantify the nonlinear work in the four production runs. The outer
iteration converges the mesh current, while each map evaluation includes
inner Picard solves for the particle orbit velocities. Within each
instability test, the larger time step requires more outer iterations
on average. The mean (maximum) outer iteration counts are
\(6.360\) (11), \(7.555\) (10), \(6.042\) (8), and \(8.008\) (9)
for TSI-0.1, TSI-0.3, WBI-0.04, and WBI-0.2, respectively.
All runs remain below the outer limit of 80.  The same figures report the particle-averaged and maximum inner
iteration counts from the final map evaluation of each accepted step.
The step averages of the particle-averaged counts are \(3.267\),
\(3.718\), \(3.683\), and \(5.038\), with maximum particle counts
of 6, 9, 4, and 8, respectively, in the same run order.
These remain below the common inner limit of 64.
The particle populations and
physical problems differ, so the counts should not be interpreted
as direct measures of relative computational cost across the two problems.

To relate the projection directly to the production example, we compare TSI calculations with identical settings except for Nyquist projection.  For the prescribed longitudinal TSI setup, the initial transverse particle velocities vanish; their subsequent behavior diagnoses departure from the intended longitudinal evolution. This criterion does not apply to the Weibel example, where transverse motion is physical and is already present in the growing eigenmode.

We compare final $x$--$v_y$ projections in \cref{fig:tsi-transverse-comparison}. Define the charge-weighted transverse RMS over all particles by
\begin{equation}
 v_{\perp,\mathrm{rms}}=
 \left[\frac{\sum_{i=1}^{N_p}|q_i|(v_{y,i}^2+v_{z,i}^2)}{\sum_{i=1}^{N_p}|q_i|}\right]^{1/2}.
 \label{eq:transverse-rms-results}
\end{equation}
At $T=60$, this quantity is $0.366930$ in the unprojected run and $2.28724\times10^{-12}$ in the projected run. The unprojected state has substantial transverse motion, while the projected state remains close to the longitudinal configuration. 
The observed suppression is consistent with the removal of Nyquist potential components that can enter the particle magnetic generator despite having a vanishing corresponding mesh derivative, as discussed in \cref{subsec:adaptive_particle_step}. Projection removes these components from the coupled field--particle update; it does not directly reset transverse particle velocities. The matched comparison demonstrates suppression for this longitudinal TSI setup, but does not by itself establish the detailed sequence through which the unprojected transverse motion develops.

\begin{figure}[!htbp]
\centering
\captionsetup[subfigure]{font=footnotesize,skip=2pt,labelfont={color=black},justification=centering}
\begin{subfigure}[t]{0.49\linewidth}
\centering
\includegraphics[width=\linewidth]{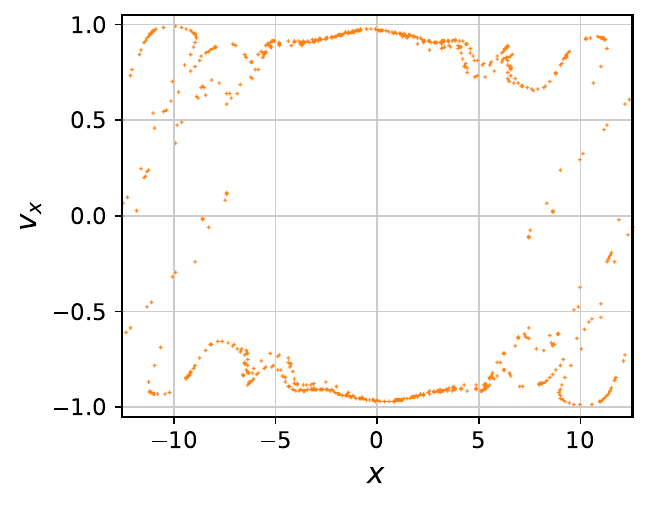}
\caption{TSI: $\Delta t=0.3$, $T=60$.}
\label{fig:phase-tsi-03}
\end{subfigure}\hfill
\begin{subfigure}[t]{0.49\linewidth}
\centering
\includegraphics[width=\linewidth]{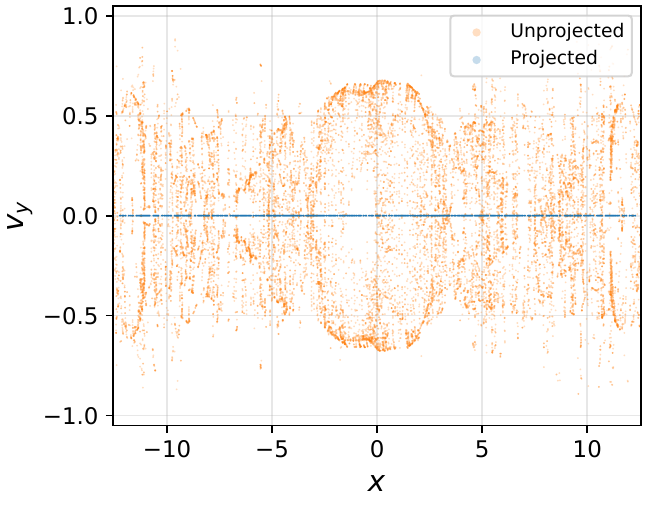}
\caption{TSI transverse motion: $\Delta t=0.1$, $T=60$.}
\label{fig:tsi-transverse-comparison}
\end{subfigure}
\par\smallskip
\begin{subfigure}[t]{0.49\linewidth}
\centering
\includegraphics[width=\linewidth]{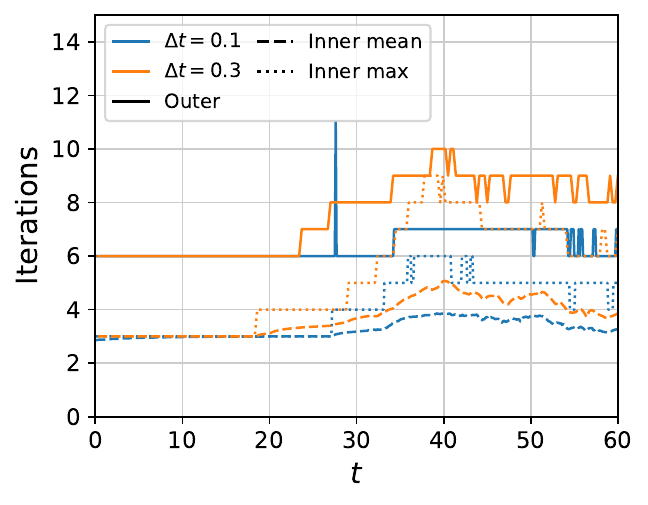}
\caption{TSI iteration counts.}
\label{fig:iteration-comparison-tsi}
\label{fig:inner-iterations-tsi}
\end{subfigure}\hfill
\begin{subfigure}[t]{0.49\linewidth}
\centering
\includegraphics[width=\linewidth]{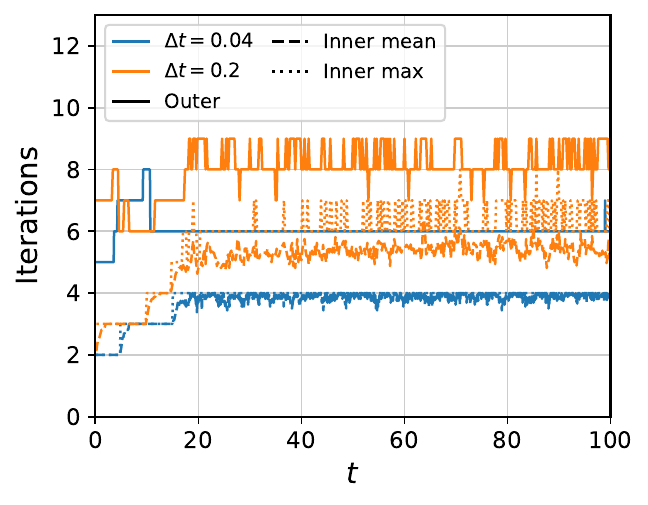}
\caption{Weibel iteration counts.}
\label{fig:iteration-comparison-wbi}
\label{fig:inner-iterations-wbi}
\end{subfigure}
\caption{Phase space, transverse motion, and nonlinear iterations. (a) Projected TSI $x$--$v_x$ portrait using a reproducible sample of 20,000 particles; reported moments use all particles. (b) Unprojected (orange) and projected (blue) TSI $x$--$v_y$ samples at the same $\Delta t=0.1$ and $T=60$. (c)--(d) Outer mesh-current counts (solid), inner particle mean (dashed), and inner particle maximum (dotted). Inner counts are from the final map evaluation and exclude earlier trials. In (c)--(d), blue and orange identify the smaller and larger field steps.}
\label{fig:phase-projection-iterations}
\end{figure}

We next examine how the particle-shape degree affects field regularity and nonlinear convergence. We compare TSI calculations with linear ($r=1$) and quadratic ($r=2$) B-splines, keeping the physical parameters, initialization, Nyquist projection, field step $\Delta t=0.1$, quadrature order, and nonlinear-solver controls fixed through $T=60$. Both calculations use the same convergence-triggered particle refinement policy: when a local particle solve fails to converge, it is retried using smaller particle time steps while the field step remains fixed. This allows us to investigate the spline-degree dependence beyond the failure of a full-step particle solve while retaining the same particle update and nonlinear iteration algorithms.
A knot is a location where the polynomial piece of the B-spline changes. The reported ``knot jump'' is the norm of the difference between reconstructed quantities evaluated at offsets \(\pm\epsilon\), where \(\epsilon=10^{-8}\Delta x\), along the probe direction from a knot. The vector-potential gradient uses the Frobenius norm. The electric-field regularity diagnostic uses $\EE_{\rm shape}=-\grad\phi_h-\UU_h$, obtained by differentiating the spline interpolant; here $\phi_h=\sum_g\phi_gS_g$ and $\UU_h=\sum_g\UU_gS_g$. This diagnostic is distinct from the gathered mesh electric field used in the particle-work balance. Maxima are taken over the sampled probe-line knots and saved diagnostic times through $T=60$. The maximum two-sided differences in the shape-derived electric field and the vector-potential gradient reach \(2.83\) and \(4.45\), respectively, for \(r=1\), and \(5.96\times10^{-8}\) and \(8.81\times10^{-8}\) for \(r=2\). These finite-offset diagnostics are consistent with discontinuous first derivatives for \(r=1\) and continuous first derivatives for \(r=2\).

Refinement is activated only in the degree-one calculation, first in the interval beginning at \(t=53.9\). The field step remains \(0.1\), while affected particles use local steps of \(0.05\) and occasionally \(0.025\). The degree-two calculation completes without activating refinement. This is convergence-triggered refinement, not an accuracy-controlled subcycling study. Such convergence-driven refinement for \(r=1\) was not needed in the nonrelativistic setting \cite{ChristliebChaconGong2026EC}. The regularity diagnostics support using the \(C^1\) quadratic B-spline without establishing that degree two is necessary or optimal.

The numerical evidence addresses three separate questions. The early mode fits assess agreement with the seeded finite-grid growing branches; energy and structural residuals assess the coupled discrete identities; and matched-time particle and mode diagnostics assess sensitivity to the field time step. The TSI pair shows close agreement in the reported longitudinal diagnostics through $T=60$. The Weibel pair conserves energy and resolves early growth at both steps, yet differs appreciably in its nonlinear state at $T=100$. Neither conservation nor suppression of Nyquist feedback alone establishes temporal accuracy.

\section{Conclusions}
\label{sec:conclusions}

We have formulated a relativistic generalized-momentum HC particle push for unstaggered potential PIC. The construction retains the generalized-momentum advantage:
the particle force does not require an explicit finite-difference
approximation of $\partial_t\AV$. The time derivative of the vector potential is represented by endpoint transformations and the orbit-discrete-gradient chain rule.  The new step splits the orbit-discrete-gradient contribution into the endpoint/chain-rule part $(\DA_i\AV)\vbar_i$ and the skew magnetic part $(\KA_i\AV)\vbar_i$, defined in \eqref{eq:D-split-skew}.  The skew part is advanced by a Higuera--Cary rotation, rather than by an additive magnetic kick.

For prescribed fields and frozen orbit data, the GM--HC
particle map preserves canonical phase-space volume. The transformation between $(\xx,\PP)$ and $(\xx,\pp)$ has determinant one, and the magnetic substep is the HC volume-preserving map with the effective magnetic field associated with \eqref{eq:D-split-skew}. This volume result is restricted to the frozen-orbit map. For the fully coupled implicit PIC solve, the same orbit construction supports the relativistic total-energy theorem: the average of the two electric-work secant
velocities is used consistently for the orbit, current, and particle work, while the CN field update supplies the opposite field-energy balance. Thus the converged GM--HC--CN fixed point conserves total relativistic energy up to nonlinear solver tolerance, orbit-quadrature error, and roundoff.
% For prescribed fields and frozen orbit data, the auxiliary GM--HC particle map preserves canonical phase-space volume: the endpoint transformations have determinant one, and the magnetic substep is the HC volume-preserving map generated by the orbit skew matrix. This result does not establish volume preservation of the converged orbit-dependent particle map or the fully coupled PIC map. For the fully coupled implicit solve, the orbit/current velocity is the average of the two electric-work secant velocities. Using this velocity consistently in the orbit, current, and particle work gives the relativistic total-energy balance, with the opposite field-work contribution supplied by the compatible CN update. Thus the converged GM--HC--CN fixed point conserves total relativistic energy up to nonlinear solver tolerance, orbit-quadrature error, and roundoff.

The solver uses an outer Anderson iteration on the mesh current and inner Picard particle solves. At convergence, particle enslavement recovers the same coupled discrete equations for the chosen discretization. For the same
Anderson history length, it reduces history storage from
$O(MN_p)$ to $O(MN_g)$. The self-adjoint Nyquist projection preserves the particle--grid work identity on the retained field space. In the matched TSI comparison, projection suppresses the substantial transverse motion observed without it and maintains the prescribed longitudinal evolution. The four main projected runs conserve total energy to within $6.18\times10^{-12}$ relative error over their reported intervals, without particle time subcycling, and their early fitted growth rates agree with the finite-grid predictions to within $0.2\%$. The TSI time-step pair gives close agreement in the reported longitudinal diagnostics through $T=60$, whereas the Weibel pair develops differences in the late-time magnetic-mode histories. The auxiliary spline comparison supports the robustness of quadratic particle shapes: convergence-triggered particle refinement is needed only for the degree-one calculation. These results distinguish conservation and nonlinear convergence from temporal accuracy. The particle-enslaved formulation provides a basis for future large-scale parallel implementations with accuracy-controlled particle subcycling and physics-based preconditioning \cite{ChenChaconBarnes2011,ChenChacon2015}.

\section*{Acknowledgments}
The authors acknowledge the use of ChatGPT for grammar
checking. The authors also thank Dr.~Lee Ricketson of LLNL for conversations about the HC pusher and energy conservation. The authors also acknowledge the Institute for Cyber-Enabled Research at Michigan State University for providing high-performance computing resources.

\bibliographystyle{siamplain}
\bibliography{ref}

\end{document}